\documentclass[reqno,11pt]{amsart}
\usepackage{amssymb}
\usepackage{graphicx}
\usepackage{amsmath}
\usepackage{mathbbol}
\usepackage[margin=1.23in]{geometry}

\usepackage[usenames, dvipsnames]{color}
\usepackage{verbatim}
\usepackage{mathrsfs}
\usepackage{bm}
\usepackage{cite}
\usepackage{color}
\allowdisplaybreaks[4]

\usepackage{subfigure}
\usepackage{pgfplots}

\numberwithin{equation}{section}

\newtheorem{theorem}{Theorem}[section]
\newtheorem{corollary}[theorem]{Corollary}
\newtheorem{lemma}[theorem]{Lemma}
\newtheorem{prop}[theorem]{Proposition}

\theoremstyle{definition}
\newtheorem{remark}[theorem]{Remark}

\theoremstyle{definition}

\theoremstyle{definition}

\makeatletter
\def\dashint{\operatorname%
{\,\,\text{\bf-}\kern-.98em\DOTSI\intop\ilimits@}}
\makeatother

\def\\det{\text{\det}}

\def\.5{\frac{1}{2}}

\newcommand{\RN}[1]{%
\textup{\uppercase\expandafter{\romannumeral#1}}%
}

\newcommand{\dist}{\text{dist}}

\renewcommand{\epsilon}{\varepsilon}

\newcounter{marnote}

\begin{document}

\title[Optimal higher derivative estimates]{Optimal higher derivative estimates for solutions of the Lam\'e system with closely spaced hard inclusions}

\author[H.J. Dong]{Hongjie Dong}
\address[H.J. Dong]{Division of Applied Mathematics, Brown University, 182 George Street, Providence, RI 02912, USA}
\email{Hongjie\_Dong@brown.edu}
\thanks{H. Dong was partially supported by the NSF under agreement DMS-2055244 and DMS-2350129.}

\author[H.G. Li]{Haigang Li}
\address[H.G. Li]{School of Mathematical Sciences, Beijing Normal University, Laboratory of Mathematics and Complex Systems, Ministry of Education, Beijing 100875, China.}
\email{hgli@bnu.edu.cn}
\thanks{H. Li was partially Supported by Beijing Natural Science Foundation (No. 1242006), the Fundamental Research Funds for the Central Universities (No. 2233200015), and National Natural Science Foundation of China (No. 12471191).}

\author[H.J. Teng]{Huaijun Teng}
\address[H.J. Teng]{School of Mathematical Sciences, Beijing Normal University, Laboratory of Mathematics and Complex Systems, Ministry of Education, Beijing 100875, China.}
\email{hjt2021@mail.bnu.edu.cn}

\author[P.H. Zhang]{PeiHao Zhang}
\address[P.H. Zhang]{School of Mathematical Sciences, Beijing Normal University, Laboratory of Mathematics and Complex Systems, Ministry of Education, Beijing 100875, China.}
\email{phzhang@mail.bnu.edu.cn}


\date{\today} 



\begin{abstract}
We investigate higher derivative estimates for the Lam\'e system with hard inclusions embedded in a bounded domain in \(\mathbb{R}^{d}\). As the distance \(\varepsilon\) between two closely spaced hard inclusions approaches zero, the stress in the narrow regions between the inclusions increases significantly. This stress is captured by the gradient of the solution. The key contribution of this paper is a detailed characterization of this singularity, achieved by deriving higher derivative estimates for solutions to the Lam\'e system with partially infinite coefficients. These upper bounds are shown to be sharp in two and three dimensions when the domain exhibits certain symmetries. To the best of our knowledge, this is the first work to precisely quantify the singular behavior of higher derivatives in the Lam\'e system with hard inclusions.
\end{abstract}

\maketitle

\section{Introduction}

We consider the Lam\'e system in linear elasticity. Let $D$ be a smooth bounded open set in $\mathbb R^{d}$, where $d\geq2$, and $D_{1}$ and $D_{2}$ be two disjoint convex smooth open subsets of $D$ that are separated by a distance $\varepsilon$ and located away from the boundary $\partial D$. Specifically,we assume
\begin{equation*}
\begin{split}
\overline{D}_{1},\overline{D}_{2}\subset D,\quad
\varepsilon:=\mbox{dist}(D_{1},D_{2})\ge 0,\quad\mbox{dist}(D_{1}\cup D_{2},\partial D)>\kappa_{0}>0,
\end{split}
\end{equation*}
where $\kappa_{0}$ is a constant independent of $\varepsilon$. For a fixed integer $m\ge0$, we assume that the $C^{m+1,\gamma}$ norms of $\partial{D}_{1}$ and $\partial{D}_{2}$ are bounded by some positive constant, independent of $\varepsilon$, ensuring that the inclusions are not too small. Set $\Omega:=D\setminus\overline{D_{1}\cup D_{2}}$ as illustrated in Figure \ref{figure:domain}. We assume that $\Omega$ and $D_{1}\cup D_{2}$ are occupied by two different homogeneous and isotropic materials with distinct Lam\'{e} constants $(\lambda, \mu)$ and $(\lambda_1, \mu_1)$. Then the elasticity tensors for the background and the inclusions can be expressed, respectively, as $\mathbb{C}^0$ and $\mathbb{C}^1$, where
$$
C_{ijkl}^0=\lambda\delta_{ij}\delta_{kl} +\mu(\delta_{ik}\delta_{jl}+\delta_{il}\delta_{jk}),~\mbox{and}~
C_{ijkl}^1=\lambda_1\delta_{ij}\delta_{kl} +\mu_1(\delta_{ik}\delta_{jl}+\delta_{il}\delta_{jk})
$$
for $i, j, k, l=1,2,\cdots,d,$ and $\delta_{ij}$ is the Kronecker symbol defined by $\delta_{ij}=0$ for $i\neq j$, $\delta_{ij}=1$ for $i=j$. 

Given a vector-valued function $\boldsymbol{\varphi}=(\boldsymbol{\varphi}^{(1)},\boldsymbol{\varphi}^{(2)},\cdots,\boldsymbol{\varphi}^{(d)})^{\mathrm{T}}$, we consider the following Dirichlet problem
\begin{align}\label{Lame}
\begin{cases}
\nabla\cdot \left((\chi_{\Omega}\mathbb{C}^0+\chi_{D_{1}\cup{D}_{2}}\mathbb{C}^1){\bf e}({\bf u})\right)=0&\hbox{in}~~D,\\
{\bf u}=\boldsymbol{\varphi} &\hbox{on}~~\partial{D},
\end{cases}
\end{align}
where ${\bf u}=({\bf u}^{(1)}, {\bf u}^{(2)},\cdots,{\bf u}^{(d)})^{\mathrm{T}}:D\rightarrow\mathbb{R}^{d}$ denotes the displacement field, ${\bf e}({\bf u}):=\frac{1}{2}(\nabla {\bf u}+(\nabla {\bf u})^{\mathrm{T}})$ is the strain tensor. Here $\chi_{D}$ is the characteristic functions of $D$. We assume the standard ellipticity conditions: $\mu>0, d\lambda+2\mu>0$, and $\mu_1>0, d\lambda_1+2\mu_1>0$. The existence and uniqueness of the solution to \eqref{Lame} have been well established. 

Significant progress has been made in the regularity theory for related partial differential equations and systems with coefficients that satisfy certain piecewise continuous conditions. We would like to draw the reader's attention to the open problems regarding higher derivative estimates on page 894 of \cite{ln}. Due to the small distance $\epsilon$, the gradient (or stress) becomes concentrated in the narrow neck region between the inclusions. From a practical perspective, it is crucial to understand whether the gradient of the solution can become arbitrarily large as the inclusions approach one another. Moreover, to obtain accurate numerical results and to effectively understand and analyze this singular behavior, higher-order derivative estimates are essential—both from an engineering standpoint and for the needs of numerical simulations. Consequently, establishing estimates for the higher-order derivatives of the solution is both imperative and essential.

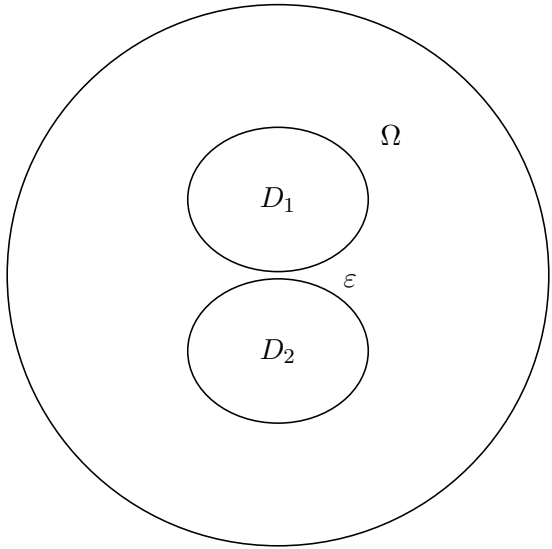
\begin{figure}[t]\centering
\begin{tikzpicture}[scale=1.2, semithick]
  \draw (0, 0) circle (3.0) +(1.25, 1.55) node {$\Omega$}; 
  \draw (0,  0.84) circle [x radius=1, y radius=0.8] node {$D_1$} +(0.80, -.90) node {$\epsilon$}  +(1.25, 0.0); 
   \draw (0, -0.84) circle [x radius=1, y radius=0.8] node {$D_2$} +(1.25, 0); 
\end{tikzpicture}
\caption{Our domain with a small distance $\epsilon=\mathrm{dist}(D_1, D_2)>0$.}\label{figure:domain}
\vspace{5pt}
\end{figure}

In recent years, many authors have obtained various results on elliptic equations and systems arising from elasticity. See, for instance, \cite{CKV,lv,ln,dz}. Chipot, Kinderlehrer, and Vergara-Caffarelli \cite{CKV} studied divergence-type uniformly elliptic systems in a domain $D\subset\mathbb{R}^{d}$ consisting of a finite number of linearly elastic, homogeneous, and parallel laminae, which models the equilibrium problem of linear laminates. Babu\u{s}ka et al. \cite{ba} computationally analyzed damage and fracture in fiber-reinforced composite materials with densely packed elastic inclusions, observing numerically that the strain tensor ${\bf e}({\bf u})$ remains bounded as the distance $\varepsilon$ approaches zero. Li and Nirenberg \cite{ln} investigated general divergence form second-order elliptic systems, including linear elasticity systems \eqref{Lame}, and affirmatively answered the question posed by the numerical results in \cite{ba}, proving that the strain tensor remains bounded as $\varepsilon$ tends to $0$. 

There have been many studies on a simplified modeling problem known as the conductivity problem, which involves a scalar elliptic equation with piecewise constant coefficients
\begin{equation}\label{scalareq}
\nabla(a_{k}(x)\nabla u_{k})=0\quad\mbox{in}~D,
\end{equation} 
where $a_{k}(x)\equiv1$ in $\Omega$ and $a_{k}(x)=k\neq1$ in $D_{1}\cup{D}_{2}$, with $0<k<\infty$.  In the anti-plane shear model, the out-of-plane elastic displacement satisﬁes the two dimensional conductivity equation \eqref{scalareq}. Therefore, the gradient estimates for $u_{k}$ are valuable for the failure analysis of composite material. Bonnetier and Vogelius \cite{bv} were the first to study a special case in $\mathbb{R}^2$, where  $D$ is a disk $B_{R_0} := \{x : |x| < R_0\}$, and $D_1$ and $D_2$ are two unit disks centered at $(0,-1)$ and $(0,1)$, such that their closure touch at the origin (i.e., $\varepsilon=0$). According to classic elliptic theory, away from the origin, $u$ is smooth in each subdomain up to the interfacial boundaries. However, the regularity near the origin is more subtle due to the geometry of the subdomains. They proved that the gradient of the solution $u_{k}$ of \eqref{scalareq} was bounded using M{\"o}bius transformations. 

Subsequently, higher regularity for solutions to equations with adjacent inclusions became an important topic of study. Significant progress has been made on equation \eqref{scalareq} with piecewise constant coefficients. Using conformal mappings, Li and Vogelius \cite{lv} proved that the solutions to \eqref{scalareq} are piecewise smooth up to interfacial boundaries, when the subdomains $D_{1}$ and $D_{2}$ are two touching unit disks in $\mathbb{R}^{2}$, and $D$ is a disk $B_{R_0}$ with sufficiently large $R_0$. Dong and Zhang \cite{dz} removed the requirement that $R_0$ be sufficiently large by constructing Green’s function and using classical Schauder estimates. For the case when $\varepsilon>0$, Dong and Li \cite{dl} utilized the Green function method to derive higher derivative estimates and demonstrated the explicit dependence of the coefficients and the distance between interfacial boundaries of inclusions. Recently, Ji and Kang \cite{JK}, as well as Dong and Yang \cite{DY}, explored higher derivative estimates for circular inclusions under different sign assumptions on relative conductivities. It is important to note that, in all these works, the dimension is always assumed to be two and the inclusions are circular. 

For the scalar equation case, the estimates in \cite{ln,lv,dz,dl,DY,JK} depend on the contrast of the coefficients $k$. If the contrast $k$ deteriorates, the situation changes signficantly, and the effect of the small distance $\varepsilon$ on the singular behavior of the gradient has been studied extensively. In particular, it was shown in various papers, such as Budiansky and Carrier \cite{BC}, Keller \cite{keller1,keller2}, and Markenscoff \cite{M}, that when $k=\infty$ in equation \eqref{scalareq}, the $L^{\infty}$-norm of $|\nabla u_{\infty}|$ tends to become unbounded as $\varepsilon$ approaches $0$. For the perfect conductivity problem ($k=\infty$), it has been established that the generic blow-up rate of $|\nabla u_{\infty}|$ is $\varepsilon^{-1/2}$ in two dimensions, $|\varepsilon\log\varepsilon|^{-1}$ in three dimensions, and $\varepsilon^{-1}$ in dimensions higher than three; see \cite{kleey,kly,akl,aklll,lyu1,bly1,bt} for more details. These blow-up rates were proven to be optimal and are independent of the shape of the inclusions, as long as the two adjacent inclusions are relatively strictly convex. To explicitly characterize the singular behavior of the gradient, the asymptotic formulas for the gradient were derived in \cite{kly2,lwx,lly,l}.  For the insulated conductivity problem ($k=0$), it was shown in \cite{akl} that the optimal blow-up rate is $\varepsilon^{-1/2}$ in two dimensions, and $\varepsilon^{-1/2+\beta(d)}$ for higher dimensions $d\geq3$ by Li and Yang \cite{LY}, Weinkove \cite{Wen}, Dong, Li and Yang \cite{DLY,DLY2}, Li and Zhao \cite{LiZhao}. For more work on the conductivity problem and related topics, we refer the reader to \cite{abtv,y1,limyun,lhg,lyu}.

In the case of the Lam{\'e} system in linear elasticity, when the Lam{\'e} parameters of inclusions degenerate to $\infty$, the optimal blow-up rates of the gradient are shown to be $\varepsilon^{-1/2}$ in two dimensions \cite{bll1,ky} and $(\varepsilon|\log\varepsilon|)^{-1}$ in three dimensions \cite{bll2}, as demostrated by providing pointwise upper bounds. See also the corresponding lower bound estimates \cite{l}. Addtional related work can be found in \cite{dx,hl,llby}. When the Lam{\'e} parameters of the inclusions degenerate to $0$, the gradient estimates for the Lam\'e systems were studied by Lim and Yu \cite{lyu} for two dimensions and Li and Yang \cite{LY} for all dimensions $n\geq2$, where an upper bound $\varepsilon^{-1/2}$ was derived. 

The main objective of this paper is to further establish higher derivative estimates that fully characterize the singular behavior when the Lam{\'e} parameters of inclusions degenerate to $\infty$ or $0$. 
For the sake of simplicity, we first translate and rotate the coordinate system so that the two points $P_1\in\partial D_1$ and $P_2\in\partial D_2$, which satisfy $\dist(P_1,P_2)=\dist(\partial D_1,\partial D_2)=\varepsilon$, are represented as $P_1=(0',\frac{\varepsilon}{2}) \in \partial D_1$ and $P_2=(0',-\frac{\varepsilon}{2}) \in \partial D_2$. There exists a constant $R$, independent of $\varepsilon$, such that the portions of $\partial D_1$ and $\partial D_2$ near the origin can be represented, respectively, by graphs 
\begin{equation*}
x_d=\frac{\varepsilon}{2}+h_1(x')\quad\text{and}\quad x_d=-\frac{\varepsilon}{2}-h_2(x')\quad \text{for}~ |x'|\leq 2R,
\end{equation*}
where $h_1$, $h_2\in C^{m+1,\gamma}(B'_{2R}(0))$ and satisfy 
\begin{align}
&-\frac{\varepsilon}{2}-h_{2}(x') <\frac{\varepsilon}{2}+h_{1}(x')\quad\mbox{for}~~ |x'|\leq 2R,\label{h1-h2}\\
&h_{1}(0)=h_2(0)=0,\quad \nabla_{x'} h_{1}(0)=\nabla_{x'} h_2(0)=0, \quad \nabla^2_{x'}(h_1+h_2)(x')\ge\kappa I_{n-1},\label{h1h1}\\
&|h_{i}(x')|\leq C|x'|^2,~|\nabla_{x'} h_{i}(x')|\leq C|x'|,~|\nabla_{x'}^k h_{i}(x')|\leq C\quad\mbox{for}~~|x'|<2R,\label{h1h14}
\end{align}
where $I_{n-1}$ denotes the $(n-1)\times(n-1)$ identity matrix, $\kappa>0$ is a constant, and $2\le k\le m+1$, $i=1,2$. 
We introduce the linear space of rigid displacement in $\mathbb{R}^{d}$, $\Psi:=\{\boldsymbol{\psi}\in C^1(\mathbb{R}^{d}; \mathbb{R}^{d})\ |\ \nabla\boldsymbol{\psi}+(\nabla\boldsymbol{\psi})^{\mathrm{T}}=0\}$ with a basis $\left\{\boldsymbol{e}_{i},~x_{j}\boldsymbol{e}_{k}-x_{k}\boldsymbol{e}_{j}:~1\leq\,i\leq\,d,~1\leq\,j<k\leq\,d\right\}$,
where $\boldsymbol{e}_{1},\cdots,\boldsymbol{e}_{d}$ denote the standard basis of $\mathbb{R}^{d}$. For fixed $\lambda$ and $\mu$, denoting the solution of (\ref{Lame}) by ${\bf u}_{\lambda_1,\mu_1}$, it was shown in \cite{bll1} that
${\bf u}_{\lambda_1,\mu_1}\rightarrow {\bf u}$ in $H^1(D; \mathbb{R}^{d})$ as $\min\{\mu_1, d\lambda_1+2\mu_1\}\rightarrow\infty$,
where ${\bf u}$ is the unique solution in $H^1(D; \mathbb{R}^{d})$ to
\begin{align}\label{maineqn}
\begin{cases}
\mathcal{L}_{\lambda, \mu}{\bf u}:=\nabla\cdot(\mathbb{C}^0{\bf e}({\bf u}))=0\quad&\hbox{in}\ \Omega,\\
{\bf u}|_{+}={\bf u}|_{-}&\hbox{on}\ \partial{D}_{i},i=1,2,\\
{\bf e}({\bf u})=0&\hbox{in}~~D_{i},i=1,2,\\
\int_{\partial{D}_{i}}\frac{\partial {\bf u}}{\partial \nu}\Big|_{+}\cdot\boldsymbol{\psi}_{\alpha}=0&i=1,2,\alpha=1,2,\cdots,\frac{d(d+1)}{2},\\
{\bf u}=\boldsymbol{\varphi}&\hbox{on}\ \partial{D},
\end{cases}
\end{align}
where $\frac{\partial {\bf u}}{\partial \nu}\big|_{+}:=(\mathbb{C}^0{\bf e}({\bf u}))\vec{n}=\lambda(\nabla\cdot {\bf u})\vec{n}+\mu(\nabla {\bf u}+(\nabla {\bf u})^{\mathrm{T}})\vec{n}$,
and $\vec{n}$ is the unit outer normal of $D_{i}$, $i=1,2$. Here and throughout this paper, the subscript $\pm$ indicates the limit from outside and inside the inclusions, respectively. The existence, uniqueness, and regularity of weak solutions to (\ref{maineqn}) can be found in the Appendix of \cite{bll1}. For $0\leq r\leq 2R$, we define the neck region between $D_{1}$ and $D_{2}$ by
\begin{equation*}
\Omega_r:=\left\{(x',x_{d})\in \Omega:-\frac{\varepsilon}{2}-h_2(x')<x_{d}<\frac{\varepsilon}{2}+h_1(x'),~|x'|<r\right\}
\end{equation*} 
with top and bottom boundaries denoted by 
$$
\Gamma^{+}_r:=\left\{(x',x_{d})\in \Omega:x_{d}=\frac{\varepsilon}{2}+h_1(x'),~|x'|<r\right\}
$$
and 
$$
\Gamma^{-}_r:=\left\{(x',x_{d})\in \Omega:x_{d}=-\frac{\varepsilon}{2}- h_2(x'),~|x'|<r\right\}.
$$

Since the maximum principle does not hold for elliptic systems, the method developed for the scalar conductivity problem cannot be directly applied to the Lam\'e system. To address this fundamental challenge, the second author, in collaboration with Bao and Li, developed an energy iteration technique and established following pointwise upper bounds of $|\nabla{\bf u}|$ in $\Omega_{R}$ in \cite{bll1,bll2}:
\begin{equation}\label{gradientest}
\begin{aligned}
    &|\nabla{\bf u}(x_{1},x_{2})|\leq\frac{C}{\sqrt{\varepsilon}+|x_{1}|}\quad\mbox{if}~d=2,\\
    &|\nabla{\bf u}(x',x_{3})|\leq\frac{C}{|\log\varepsilon|(\varepsilon+|x'|^{2})}+\frac{C}{(\varepsilon+|x'|^{2})^{1/2}}\quad\mbox{if}~d=3.
\end{aligned}
\end{equation}
These pointwise estimates clearly illustrate the stress concentration. However, to accurately simulate the singular behavior of $\nabla{\bf u}$, analyzing the gradient blow-up alone is insufficient for designing an effective numerical scheme. Therefore, it becomes essential to examine the singularities in the higher derivatives of the solution to \eqref{maineqn} in the neck region $\Omega_{R}$.  

A crucial element in proving \eqref{gradientest} in \cite{bll1,bll2} involves the construction of an auxiliary function in the narrow region $\Omega_{R}$. This function is formed by combining the Keller-type function with the basis function $\boldsymbol{\psi}_{\alpha}$. In the context of the regularity theory of partial differential equations, a natural question arises: Can we construct precise auxiliary functions to effectively capture all the main singular terms of the higher derivatives $\nabla^{m}{\bf u}$ in the narrow region $\Omega_{R}$ for $m\geq2$? To the best of our knowledge, this question remains unanswered. In this paper, we provide a positive answer. While the analogous result for the scalar perfect conductivity problem was recently addressed in \cite{LLYZ}, the linear elasticity problem \eqref{maineqn} presents greater technical challenges than the scalar case.

According to the standard theory for elliptic systems, we know that $\|\nabla^{m} {\bf u}\|_{L^{\infty}(\Omega\setminus\Omega_{R})}$ is bounded for $m\ge 1$. Our first main result focuses on higher derivative estimates for problem \eqref{maineqn} in the narrow region $\Omega_{R}$ in two dimensions.
\begin{theorem}\label{mainth}
Assume that $\Omega$, $D_1$, and $D_2$ are defined as above in dimension $d=2$, and $\boldsymbol{\varphi} \in C^{m,\alpha}(\partial D,\mathbb{R}^2)$ for an integer $m\ge 1$ and $0<\alpha<1$. Let ${\bf u}\in H^1(D,\mathbb{R}^2)\cap C^{m}(\bar{\Omega},\mathbb{R}^2)$ be a solution to \eqref{maineqn}. Then for $0\le \varepsilon<1/4$, we have
\begin{align}\label{mthm_equ1}
|\nabla^{m}{\bf u}(x_{1},x_{2})|\leq\frac{C}{(\varepsilon+x_1^2)^{m/2}}\quad\mbox{for}~(x_{1},x_{2})\in\Omega_R.
\end{align} 
\end{theorem}

\begin{remark} Note that if $m=1$, estimate \eqref{mthm_equ1} is consistent with  \eqref{gradientest} previously derived in \cite{bll1,LX23}. These upper bounds for higher derivatives are proved to be sharp in Theorem \ref{thm-5.12} when the domain and boundary data satisfy additional symmetry conditions. An asymptotic expansion of higher derivatives is also established to precisely quantify this singularity in Section \ref{sec5}; see Theorem \ref{mainth-5.3}. 
\end{remark}

Our method can be extended to deal with higher dimensional cases.

\begin{theorem}\label{mainth3d}
Assume that $\Omega$, $D_1$, and $D_2$ are defined as above in dimension $d\ge3$, and $\boldsymbol{\varphi} \in C^{m,\alpha}(\partial D,\mathbb{R}^d)$ for an integer $m\ge 1$ and $0<\alpha<1$. Let ${\bf u}\in H^1(D,\mathbb{R}^d)\cap C^{m}(\bar{\Omega},\mathbb{R}^d)$ be a solution to \eqref{maineqn}. Then for $0\le \varepsilon<1/4$, we have
\begin{align*}
|\nabla^{m}{\bf u}(x',x_{3})|\leq 
\frac{C}{|\log\varepsilon|(\varepsilon+|x'|^2)^{\frac{m+1}{2}}}+\frac{C}{(\varepsilon+|x'|^2)^{m/2}}\quad\mbox{for}~(x',x_{3})\in\Omega_R,
\end{align*} 
and
\begin{align*}
|\nabla^{m}{\bf u}(x',x_{d})|\leq 
\frac{C}{(\varepsilon+|x'|^2)^{\frac{m+1}{2}}}\quad\mbox{for}~(x',x_{d})\in\Omega_R, ~d\ge 4.
\end{align*} 
\end{theorem}
\begin{remark}
The optimality of the blow-up rate $|\log\varepsilon|^{-1}\varepsilon^{-(m+1)/2}$ in dimension three is shown in Theorem \ref{thm-5.12} and an asymptotic formula is given in Theorem \ref{mainth23d}. 
\end{remark}

Finally, we establish an upper bound estimate for the higher-order derivatives of the solution to the Lam\'e system with two closely spaced holes:
\begin{align}\label{maineqn222}
	\begin{cases}
		\mathcal{L}_{\lambda, \mu}{\bf u}:=\nabla\cdot(\mathbb{C}^0{\bf e}({\bf u}))=0\quad&\hbox{in}\ \Omega,\\
	\frac{\partial {\bf u}}{\partial \nu}\Big|_{+}=0&\hbox{on}\ \partial{D}_{i},~i=1,2,\\
		{\bf u}=\boldsymbol{\varphi}&\hbox{on}\ \partial{D},
	\end{cases}
\end{align}
where $\nu$ denotes the unit outward normal vector on $\partial{D}_{i}$.
\begin{theorem}\label{mainthinsl}
	Assume that $\Omega$, $D_1$, and $D_2$ are defined as in Theorem \ref{mainth}. Let ${\bf u}\in H^1(D,\mathbb{R}^d)\cap C^{m}(\bar{\Omega},\mathbb{R}^d)$ be a solution to \eqref{maineqn222}. Then for $0\le \varepsilon<1/4$, there holds 
	\begin{align}\label{inslut1}
		|\nabla^{m}{\bf u}(x',x_{d})|\leq C\|{\bf u}\|_{L^\infty (\Omega_{R})}(\varepsilon+|x'|^2)^{-m/2}\quad\mbox{for}~(x',x_{d})\in\Omega_R.
	\end{align} 
\end{theorem}

The remaining part of this paper is structured as follows. In Section \ref{sec2}, we develop some new ingredients in dimension two for the Lam\'e systems with specified Dirichlet boundary data $\boldsymbol{\psi}_{\alpha}$ on $\partial{D}_{1}$ and $\partial{D}_{2}$, as given in \eqref{equ_v1} below. For each $\boldsymbol{\psi}_{\alpha}$, we construct $m$ auxiliary vector-valued functions ${\bf v}_{\alpha}^{l}(x)=(({\bf v}_{\alpha}^{l})^{(1)}(x),
({\bf v}_{\alpha}^{l})^{(2)}(x))^{T}$, $1\leq\,l\leq m$, $\alpha=1,2,3$, and then demonstrate that these auxiliary functions can effectively capture all the singularity of $\nabla^m {\bf u}_{1\alpha}$ up to $O(1)$. The construction of these functions is intricate and depends on the coefficients of Lam\'e system. Using an energy iteration technique, we reduce the estimates of $\|\nabla^{m}({\bf u}_{1\alpha}-\sum_{l=1}^{m}{\bf v}_{\alpha}^{l})\|_{L^\infty}$ to estimating $|\mathcal{L}_{\lambda,\mu}\sum_{l=1}^{m}{\bf v}_{\alpha}^{l}|$, which decreases when $m$ increases. This approach is entirely novel and significantly improves the previous results, extending the gradient estimates to higher-order derivative estimates. In Section \ref{sec3}, we prove Theorems \ref{mainth} by utilizing of the estimates proved in Section \ref{sec2}. In Section \ref{sec4}, we provide details on the construction of the auxiliary function vectors in dimension three, highlight the differences from the two-dimensional case, and prove Theorem \ref{mainth3d}. In Section \ref{sec5}, under some additional symmetry assumptions on the domain and boundary data, we establish the precise asymptotic characterizations of the higher derivatives and prove the optimality of the upper bounds by deriving a corresponding lower bound in dimensions two and three. In Section \ref{sec6}, we establish an upper bound estimate for the higher-order derivatives of the solution to the Lam\'e system with two closely spaced holes. A structured framework for the energy method is provided in the Appendix.

In the proofs of our main theorems below, we assume $\varepsilon>0$, though the arguments remains valid when $\varepsilon=0$, as the implicit constants do not blow up as $\varepsilon\to 0$.

\section{New Ingredients for the Problem in Dimension Two}\label{sec2} 

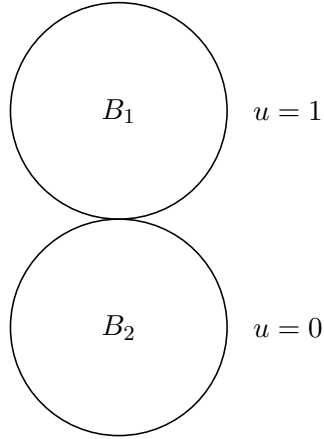
\begin{figure}[t]\centering
\begin{tikzpicture}[scale=1.8, semithick]
  \draw (0, 0.8) circle [x radius=0.8, y radius=0.8] node {$B_1$} +(1.25, 0) node {$u=1$};
  \draw (0, -0.8) circle [x radius=0.8, y radius=0.8] node {$B_2$} +(1.25, 0) node {$u=0$};
\end{tikzpicture}
\caption{The domain $\mathbb{R}^{2}\backslash \overline{B_1\cup B_2}$.}\label{figure2}
\vspace{5pt}
\end{figure}
Let $D_1=B_{1}(0,1)$ and $D_2=B_{2}(0,-1)$ be two unit disks centered at $(0,1)$ and $(0,-1)$, so their closure touch at the origin (i.e., $\epsilon=0$), as Figure \ref{figure2}. For the Dirichlet problem
\begin{equation}\label{laplace}
\begin{cases}
\Delta u=0&\mathrm{in}~\mathbb{R}^{2}\setminus\overline{D_{1}\cup{D}_{2}},\\
\quad u=1&\mathrm{on}~\partial{D}_{1}\setminus\{0\},\\
\quad u=0&\mathrm{on}~\partial{D}_{2},
\end{cases}
\end{equation}
by using conformal transform, it is well known that $u=\frac{x_{2}}{x_{1}^{2}+x_{2}^{2}}+\frac{1}{2}$ is a singular solution. Moreover, any solution of \eqref{laplace} can be written as the sum of $u$ and a function which has bounded derivatives of any order. In other words, $u$ captures the singular behavior of any solution. A natural question is: when $\epsilon>0$, can we find explicit approximate functions, which captures the singular behavior of the solution to \eqref{laplace} in the narrow region $\Omega_R$? How about the Dirichlet problem for the Lam\'e system?

 Assume that $D_{1}$ and $D_2$ have a small separation distance $\varepsilon>0$. In this section, we will establish higher derivative estimates of the solution to the following Dirichlet problems of Lam\'e system in dimension two. Let ${\bf u}_{1\alpha}$, $\alpha=1,2,3$, be the solution to 
\begin{equation}\label{equ_v1}
\begin{cases}
\mathcal{L}_{\lambda,\mu}{\bf u}_{1\alpha}=0&\mathrm{in}~\Omega,\\
{\bf u}_{1\alpha}=\boldsymbol{\psi}_{\alpha}&\mathrm{on}~\partial{D}_{1},\\
{\bf u}_{1\alpha}=0&\mathrm{on}~\partial{D_{2}}\cup\partial{D},
\end{cases}
\end{equation}
where 
$${\boldsymbol\psi}_{1}=\begin{pmatrix}
1 \\
0
\end{pmatrix},\quad
{\boldsymbol\psi}_{2}=\begin{pmatrix}
0\\
1
\end{pmatrix},\quad
{\boldsymbol\psi}_{3}=\begin{pmatrix}
x_{2}\\
-x_{1}
\end{pmatrix}.$$
The key steps of our proof are to construct a series of explicit functions, ${\bf v}_{\alpha}^{l}(x)=(({\bf v}_{\alpha}^{l})^{(1)}(x),$ $
({\bf v}_{\alpha}^{l})^{(2)}(x))^{T}$, to approximate the solutions ${\bf u}_{1\alpha}$, $\alpha=1,2,3$, and to prove that the higher derivatives of these functions, $\nabla^{m}\sum_{l=1}^{m}{\bf v}_{\alpha}^{l}$, can fully capture the singularity of $\nabla^{m}{\bf u}_{1\alpha}$ in $\Omega_R$, up to $O(1)$. 

Recall that
\begin{equation*}
(\mathcal{L}_{\lambda,\mu}{\bf u})^{(1)}=\mu\Delta {\bf u}^{(1)}+(\lambda+\mu)\partial_{x_1}(\nabla\cdot{\bf u}),
\end{equation*}
\begin{equation*}
	(\mathcal{L}_{\lambda,\mu}{\bf u})^{(2)}=\mu\Delta {\bf u}^{(2)}+(\lambda+\mu)\partial_{x_2}(\nabla\cdot{\bf u}).
\end{equation*}

\subsection{Estimates of $\nabla^{m}{\bf u}_{11}$}

The following theorem about the singular behavior of $\nabla^{m}{\bf u}_{1\alpha}$ is an important step to prove Theorems \ref{mainth}, which is also an essential improvement of previous results on gradient estimates in \cite{bll1,LX23}. 

\begin{theorem}\label{thm3.1}
Under the same assumption as in Theorem \ref{mainth}, let ${\bf u}_{1\alpha}$ be the solution to \eqref{equ_v1} for $\alpha=1, 2$. Then for sufficiently small $0<\varepsilon<1/2$ and for $m\ge 1$, we have 
\begin{align*}
	|\nabla^{m}{\bf u}_{1\alpha}(x) | \le \frac{C}{(\varepsilon+|x_{1}|^{2})^{\frac{m+1}{2}}}\quad\text{for}~x\in\Omega_{R}.
\end{align*}
\end{theorem}
Denote the vertical distance between $D_{1}$ and $D_{2}$ by 
$$\delta(x_1):=\varepsilon+h_{1}(x_1)+h_{2}(x_1)\quad\mbox{for}~|x_{1}|\leq2R.$$

\begin{proof}[Proof of Theorem \ref{thm3.1}]
We will prove Theorem \ref{thm3.1} by induction for $m\geq1$. We only prove the case when
$\alpha=1$, since the case when $\alpha=2$ is the same.  

We will construct a sequence of explicit functions, ${\bf v}_{\alpha}^{l}(x)=(({\bf v}_{\alpha}^{l})^{(1)}(x),
({\bf v}_{\alpha}^{l})^{(2)}(x))^{T}$, which are polynomials with respect to $x_2$, to make $|\mathcal{L}_{\lambda,\mu}\sum_{l=1}^{m}{\bf v}_{\alpha}^{l}|$ as small as possible. Then by the energy iteration method presented in the Appendix, which is a variation one  developed in \cite{bll1,bll2}, we demonstrate that $\nabla^{m}\sum_{l=1}^{m}{\bf v}_{\alpha}^{l}$ can capture all singular terms in $\nabla^{m}{\bf u}_{1\alpha}$.

{\bf Step 1. $m=1$ (the gradient estimates).} 
We first construct an auxiliary function ${\bf v}_{1}^{1}(x)=\begin{pmatrix}
({\bf v}_{1}^{1})^{(1)}(x),
({\bf v}_{1}^{1})^{(2)}(x)
\end{pmatrix}^{T}$ satisfying the boundary condition ${\bf v}_{1}^{1}(x)={\bf u}_{11}(x)={\boldsymbol\psi}_{1}$ on $\Gamma^+_{2R}$,  ${\bf v}_{1}^{1}(x)={\bf u}_{11}(x)=0$ on $\Gamma^-_{2R}$. Precisely,
\begin{equation*}
({\bf v}_{1}^{1})^{(1)}(x)=\frac{x_2+\varepsilon/2+h_2(x_1)}{\delta(x_1)}\quad\text{for}~x\in\Omega_{2R}
\end{equation*}
and
\begin{equation}\label{def_v112}
	({\bf v}_{1}^{1})^{(2)}(x)=\frac{\lambda+\mu}{\lambda+2\mu}\int_{-\varepsilon/2-h_2(x_1)}^{\varepsilon/2+h_1(x_1)} G(y,x_2)\partial_{x_1x_2} ({\bf v}_{1}^{1})^{(1)}(x_1,y) dy\quad\text{for}~x\in\Omega_{2R},
\end{equation}
where 
\begin{align}\label{defgreenfunction}
	G(y, x_{2})=\frac{1}{\delta(x_1)}\begin{cases}
		(\frac{\varepsilon}{2}+h_1(x_1)-x_{2})(y+\frac{\varepsilon}{2}+h_2(x_1)), \quad -\frac{\varepsilon}{2}-h_2(x_1)\le y\le x_{2},\\
		(\frac{\varepsilon}{2}+h_1(x_1)-y)(x_{2}+\frac{\varepsilon}{2}+h_2(x_1)), \quad x_{2}\le y\le \frac{\varepsilon}{2}+h_1(x_1)
	\end{cases}
\end{align}
is the Green function. Clearly, 
\begin{equation}\label{def_v112gx}
	(\lambda+2\mu)\partial_{x_2x_2}({\bf v}_{1}^{1})^{(2)}=-(\lambda+\mu)\partial_{x_1x_2} ({\bf v}_{1}^{1})^{(1)} \quad\text{in}~\Omega_{2R}.
\end{equation}

A direct calculation yields, in $\Omega_{2R}$,
\begin{equation}\label{nabla_x2v111}
\begin{split}
&|\partial_{x_1} ({\bf v}_{1}^{1})^{(1)}|\leq C\delta(x_1)^{-1/2},\quad |\partial_{x_2} ({\bf v}_{1}^{1})^{(1)}|\leq C\delta(x_1)^{-1},\\
& |\partial_{x_1x_1} ({\bf v}_{1}^{1})^{(1)}|\leq C\delta(x_1)^{-1},\quad
 |\partial_{x_1x_2} ({\bf v}_{1}^{1})^{(1)}|\leq C\delta(x_1)^{-3/2},\quad \partial_{x_2x_2} ({\bf v}_{1}^{1})^{(1)}=0.
\end{split}
\end{equation}
For the $k$-th order derivative, $k\le m+1$, we also have
\begin{align}\label{kthdv111}
 |\partial_{x_1}^k ({\bf v}_{1}^{1})^{(1)}|\leq C\delta(x_1)^{-k/2},\quad
|\partial_{x_1}^{k-1}\partial_{x_2} ({\bf v}_{1}^{1})^{(1)}|\leq C\delta(x_1)^{-(k+1)/2}\quad\text{in}~\Omega_{2R}.
\end{align}

By \eqref{h1h14} and \eqref{defgreenfunction},  
\begin{equation}\label{estgreen}
\begin{split}
&|G(y,x_2)|\leq C\delta(x_1),\\
&|\partial_{x_{2}}G(y,x_2)|\leq C,\quad |\partial_{x_{1}}G(y,x_2)|\leq C\delta(x_1)^{1/2},\\
&|\partial_{x_{1}}^kG(y,x_2)|\leq C\delta(x_1)^{-(k-2)/2},\quad 1_{y\neq x_2}\cdot\partial_{x_{2}}^kG(y,x_2)=0\quad\mbox{for}~ 2\le k\le m+1.
\end{split}
\end{equation}
By using \eqref{def_v112}, \eqref{nabla_x2v111}, and \eqref{estgreen}, we have in $\Omega_{2R}$,
\begin{align}\label{nabla_x2v112}
	|({\bf v}_{1}^{1})^{(2)}|\leq C\delta(x_1)\Big|\max_{-\varepsilon/2-h_2(x_1)\leq y\leq \varepsilon/2+h_1(x_1)}G(y,x_2)\Big||\partial_{x_1x_2} ({\bf v}_{1}^{1})^{(1)}|\leq C\delta(x_1)^{1/2}.
\end{align}
Similarly, by \eqref{h1h14}, \eqref{def_v112}, \eqref{kthdv111}, and \eqref{estgreen}, 
\begin{equation}\label{nabla_x2v1121}
\begin{split}
&|\partial_{x_1} ({\bf v}_{1}^{1})^{(2)}|\leq C\delta(x_1)^2|\partial_{x_1}^2\partial_{x_2} ({\bf v}_{1}^{1})^{(1)}|+C\delta(x_1)^{3/2}|\partial_{x_1}\partial_{x_2} ({\bf v}_{1}^{1})^{(1)}|\leq C,\\
&|\partial_{x_2} ({\bf v}_{1}^{1})^{(2)}|\leq C\delta(x_1)|\partial_{x_1x_2} ({\bf v}_{1}^{1})^{(1)}|\leq C\delta(x_1)^{-1/2}
\end{split}\quad\text{in}~\Omega_{2R}.
\end{equation}
For the second-order derivative, by virtue of \eqref{def_v112},  \eqref{defgreenfunction}, \eqref{kthdv111}, and \eqref{estgreen}, 
\begin{equation}\label{2jdsfzhsv11}
	|\partial_{x_1x_1} ({\bf v}_{1}^{1})^{(2)}|\leq C\delta(x_1)^{-1/2},\quad 	|\partial_{x_1x_2} ({\bf v}_{1}^{1})^{(2)}|\leq C\delta(x_1)^{-1}\quad\text{in}~\Omega_{2R}.
\end{equation}
Similarly, for $k\le m+1$, 
\begin{align}\label{v112dgj2d1}
	|\partial_{x_1}^k ({\bf v}_{1}^{1})^{(2)}|\leq C\delta(x_1)^{-(k-1)/2},\quad
	|\partial_{x_1}^{k-1}\partial_{x_2}  ({\bf v}_{1}^{1})^{(2)}|\leq C\delta(x_1)^{-k/2}\quad\text{in}~\Omega_{2R}.
\end{align}
In addition, it follows from \eqref{def_v112gx} and \eqref{kthdv111} that
\begin{align}\label{v112dgj2d}
	|\partial_{x_1}^{k-2}\partial_{x_2}^2  ({\bf v}_{1}^{1})^{(2)}|\leq C|\partial_{x_1}^{k-1}\partial_{x_2} ({\bf v}_{1}^{1})^{(1)}|\leq C\delta(x_1)^{-(k+1)/2}\quad\text{in}~\Omega_{2R},
\end{align}
and, for $s\ge 3$, by \eqref{nabla_x2v111}, we have
\begin{align}\label{v112dgj2dk}
\partial_{x_2}^s  ({\bf v}_{1}^{1})^{(2)}=\frac{\lambda+\mu}{\lambda+2\mu}\partial_{x_1}\partial_{x_2}^{s-1}({\bf v}_{1}^{1})^{(1)}=0\quad\text{in}~\Omega_{2R}.
\end{align}

We denote 
$${\bf f}_{1}^{1}(x):=\mathcal{L}_{\lambda,\mu}{\bf v}_{1}^{1}(x).$$ Then, by using \eqref{def_v112gx} and \eqref{nabla_x2v111}, 
\begin{align*}
({\bf f}_{1}^{1})^{(1)}&=\big(\mathcal{L}_{\lambda,\mu}{\bf v}_{1}^{1}\big)^{(1)}  =\mu\Delta({\bf v}_{1}^{1})^{(1)}+(\lambda+\mu)\partial_{x_1}\big(\nabla\cdot{\bf v}_{1}^{1}\big) \\
&=(\lambda+2\mu)\partial_{x_1x_1}({\bf v}_{1}^{1})^{(1)}+(\lambda+
\mu)\partial_{x_1x_2}({\bf v}_{1}^{1})^{(2)},\\
({\bf f}_{1}^{1})^{(2)}&=\big(\mathcal{L}_{\lambda,\mu}{\bf v}_{1}^{1}\big)^{(2)}=\mu\Delta({\bf v}_{1}^{1})^{(2)}+(\lambda+\mu)\partial_{x_2}\big(\nabla\cdot{\bf v}_{1}^{1} \big) \\
&=(\lambda+2\mu)\partial_{x_2x_2}({\bf v}_{1}^{1})^{(2)}+(\lambda+\mu)
\partial_{x_1x_2}({\bf v}_{1}^{1})^{(1)}+\mu\partial_{x_1x_1}({\bf v}_{1}^{1})^{(2)}\\
&=\mu\partial_{x_1x_1}({\bf v}_{1}^{1})^{(2)}.
\end{align*}
Thus, it follows from \eqref{nabla_x2v111} and \eqref{2jdsfzhsv11} that
\begin{align*}
	|({\bf f}_{1}^{1})^{(1)}(x)|\leq C\delta(x_1)^{-1},\quad |({\bf f}_{1}^{1})^{(2)}(x)|\leq  C\delta(x_1)^{-1/2}\quad\text{in}~\Omega_{2R}.
\end{align*}
Hence,
\begin{align}\label{m=1}
	|{\bf f}_{1}^{1}(x)| \leq|({\bf f}_{1}^{1})^{(1)} (x)| +|({\bf f}_{1}^{1})^{(2)}(x)|\leq C\delta(x_1)^{-1}\quad\text{in}~\Omega_{2R}.
\end{align}
 Moreover, for $1\le k\le m-1$, by \eqref{kthdv111}, \eqref{v112dgj2d1}, and \eqref{v112dgj2d}, we have
\begin{align*}
	|\partial_{x_1}^k({\bf f}_{1}^{1})^{(1)}|\leq C\delta(x_1)^{-(k+2)/2},\quad |\partial_{x_1}^{k-1}\partial_{x_2}({\bf f}_{1}^{1})^{(1)}|\leq C\delta(x_1)^{-(k+3)/2}\quad\text{in}~\Omega_{2R}.
\end{align*}

By applying Proposition \ref{propnew} and using estimate \eqref{m=1}, we obtain
\begin{align*}
\|\nabla\big({\bf u}_{11}-{\bf v}_{1}^{1}\big)\|_{L^\infty(\Omega_{\delta(x_1)/2}(x_1))}\leq C.
\end{align*}
Thus, Theorem \ref{thm3.1} holds true for $m=1$,
\begin{equation*}
|\nabla {\bf u}_{11}|\leq |\nabla {\bf v}_{1}^{1}|+C\leq C\delta(x_1)^{-1}\quad\text{in}~\Omega_{2R}.
\end{equation*}

Here, we emphasize the crucial role of the auxiliary function \(({\bf v}_{1}^{1})^{(2)}\). Define \(\tilde{\bf v}_1^1(x) = ({\bf v}_{1}^{1})^{(1)}(x){\boldsymbol\psi}_{1}\). Notably, the second derivative term \(\partial_{x_2x_2}({\bf v}_{1}^{1})^{(2)}\) cancels out the leading term in \(\mathcal{L}_{\lambda, \mu} \tilde{\bf v}_1^1(x)\), rendering \( |\mathcal{L}_{\lambda, \mu} {\bf v}_{1}^{1}(x)| \) less singular than \( |\mathcal{L}_{\lambda, \mu} \tilde{\bf v}_1^1(x)| \). This cancellation enables us to capture all singular terms in \(\nabla {\bf u}_{11}\) up to \(O(1)\), marking a significant improvement over the previous result in \cite[Proposition 3.2]{bll1}, which established that \( |\partial_{x_2}({\bf u}_{11} - \tilde{\bf v}_1^1)(x)| \leq \delta(x_1)^{-1/2} \). It also confirms that \(\partial_{x_2} ({\bf v}_{1}^{1})^{(2)}\) has a singularity of order \(\delta(x_1)^{-1/2}\). This demonstrates the necessity of constructing both \(({\bf v}_{1}^{1})^{(1)}\) and \(({\bf v}_{1}^{1})^{(2)}\) simultaneously. In the following sections, we will extend this approach inductively to construct additional auxiliary functions, allowing us to identify all singular terms in \(\nabla^m {\bf u}_{11}(x)\) up to \(O(1)\).

{\bf Step 2. $m\geq2$ (the higher derivatives estimates).} We denote 
$$
{\bf f}_1^{l}(x):=\sum_{j=1}^{l}\mathcal{L}_{\lambda,\mu}{\bf v}_{1}^{j}(x)= {\bf f}_1^{l-1}(x)+\mathcal{L}_{\lambda,\mu}{\bf v}_{1}^{l}(x)\quad\quad\mbox{for}~2\leq\,l\leq m.
$$ 
We choose ${\bf v}_{1}^{l}(x)=\begin{pmatrix}
({\bf v}_{1}^{l})^{(1)}(x),
({\bf v}_{1}^{l})^{(2)}(x)
\end{pmatrix}^{T}$ satisfying ${\bf v}_{1}^{l}(x)=0$ on $\Gamma^+_{2R}\cup\Gamma^-_{2R}$, and in $\Omega_{2R}$,
\begin{align}
({\bf v}_{1}^{l})^{(1)}(x)&=\frac{1}{\mu}\int_{-\varepsilon/2-h_2(x_1)}^{\varepsilon/2+h_1(x_1)} G(y,x_2) ({\bf f}_{1}^{l-1})^{(1)}(x_1,y) dy,
\label{def_v1l1}\\
({\bf v}_{1}^{l})^{(2)}(x)&=\frac{1}{\lambda+2\mu}\int_{-\varepsilon/2-h_2(x_1)}^{\varepsilon/2+h_1(x_1)} G(y,x_2)\Big( ({\bf f}_{1}^{l-1})^{(2)}+(\lambda+\mu)\partial_{x_1x_2}({\bf v}_{1}^{l})^{(1)}\Big)(x_1,y) dy.\label{def_v1l2}
\end{align}

It is easy to verify that 
$$\sum_{l=1}^{m}{\bf v}_{1}^{1}(x)={\bf u}_{11}(x)={\boldsymbol\psi}_{1}\quad\mbox{ on}~ \Gamma^+_{2R},\quad\sum_{l=1}^{m}{\bf v}_{1}^{1}(x)={\bf u}_{11}(x)=0\quad\mbox{ on}~\Gamma^-_{2R},$$ and
\begin{align}\label{fzl1}
\begin{split}
	\mu\partial_{x_2x_2}({\bf v}_{1}^{l})^{(1)}&=-({\bf f}_{1}^{l-1})^{(1)},\\
	(\lambda+2\mu)\partial_{x_2x_2}({\bf v}_{1}^{l})^{(2)}&=-({\bf f}_1^{l-1})^{(2)}-(\lambda+\mu)\partial_{x_1x_2}({\bf v}_{1}^{l})^{(1)}
\end{split}\quad\text{in}~\Omega_{2R}.
\end{align}
Thus, we have
\begin{align}\label{fqcx12d}
	({\bf f}_1^l)^{(1)}&=({\bf f}_1^{l-1})^{(1)}+\mu\partial_{x_2x_2}({\bf v}_{1}^{l})^{(1)}+(\lambda+2\mu)\partial_{x_{1}x_1}({\bf v}_1^1)^{(1)}+(\lambda+\mu)\partial_{x_1x_2}({\bf v}_{1}^{l})^{(2)}\nonumber\\
	&=(\lambda+2\mu)\partial_{x_{1}x_1}({\bf v}_1^l)^{(1)}+(\lambda+\mu)\partial_{x_1x_2}({\bf v}_{1}^{l})^{(2)},\\
	({\bf f}_1^l)^{(2)}&=({\bf f}_1^{l-1})^{(2)}+(\lambda+\mu)\partial_{x_1x_2}({\bf v}_{1}^{l})^{(1)}+(\lambda+2\mu)\partial_{x_2x_2}({\bf v}_{1}^{l})^{(2)}+\mu\partial_{x_1x_1}({\bf v}_{1}^{l})^{(2)}\nonumber\\
	&=\mu\partial_{x_1x_1}({\bf v}_{1}^{l})^{(2)}.\label{fqcx22d}
\end{align}
We will inductively prove the following estimates: for $j\geq1$, $j=1,2,\dots,m$,
\begin{equation}\label{xydgj2d1}
|({\bf v}_{1}^{j})^{(1)}(x) | \le C\delta(x_1)^{j-1},\quad|({\bf v}_{1}^{j})^{(2)}(x) | \le C\delta(x_1)^{(2j-1)/2}\quad\text{in}~\Omega_{2R},
\end{equation}
and for $ k\ge 0$, 
\begin{align}\label{xydgj2d2}
\begin{split}
&|\partial_{x_1}^k\partial_{x_2}^s({\bf v}_{1}^{j})^{(1)}(x) | \le C\delta(x_1)^{\frac{2j-2s-k-2}{2}}~\text{for}~0\le s\le 2j-1,\\
&\partial_{x_2}^{s}({\bf v}_{1}^{j})^{(1)}(x)=0~\text{for}~s\ge 2j,\\
&|\partial_{x_1}^k\partial_{x_2}^s({\bf v}_{1}^{j})^{(2)}(x) | \le C\delta(x_1)^{\frac{2j-2s-k-1}{2}}~\text{for}~0\le s\le 2j,\\
&\partial_{x_2}^{s}({\bf v}_{1}^{j})^{(2)}(x)=0~\text{for}~s\ge 2j+1,
\end{split}\quad\text{in}~\Omega_{2R}.
\end{align}

Indeed, from \eqref{nabla_x2v111}, \eqref{kthdv111}, and \eqref{nabla_x2v112}--\eqref{v112dgj2dk}, we have \eqref{xydgj2d1} and \eqref{xydgj2d2} for $j=1$. Assuming that \eqref{xydgj2d1} and \eqref{xydgj2d2} hold for $j=l-1$ with $l\ge 2$, that is
\begin{equation*}
	|\big({\bf v}_{1}^{(l-1)}\big)^{(1)}(x) | \le C\delta(x_1)^{l-2},\quad|\big({\bf v}_{1}^{(l-1)}\big)^{(2)}(x) | \le C\delta(x_1)^{(2l-3)/2},
\end{equation*}
and for $ k\ge 0$, 
\begin{align}\label{xydgj2d2js}
\begin{split}
&|\partial_{x_1}^k\partial_{x_2}^s\big({\bf v}_{1}^{(l-1)}\big)^{(1)}(x) | \le C\delta(x_1)^{\frac{2l-2s-k-4}{2}}~\text{for}~0\le s\le 2l-3,\\
&\partial_{x_2}^{s}({\bf v}_{1}^{(l-1)})^{(2)}(x)=0~\text{for}~s\ge 2l-2,\\
&|\partial_{x_1}^k\partial_{x_2}^s\big({\bf v}_{1}^{(l-1)}\big)^{(2)}(x) | \le C\delta(x_1)^{(2l-2s-k-3)/2}~\text{for}~0\le s\le 2l-2,\\
&\partial_{x_2}^{s}({\bf v}_{1}^{(l-1)})^{(2)}(x)=0~\text{for}~s\ge 2l-1.
\end{split}
\end{align}
Then, by \eqref{fqcx12d}, \eqref{fqcx22d}, and \eqref{xydgj2d2js}, 
\begin{equation}\label{xydgjf}
\begin{split}
&| ({\bf f}_{1}^{l-1})^{(1)}(x)|\leq C|\partial_{x_1x_1}({\bf v}_{1}^{(l-1)})^{(1)}|+C|\partial_{x_1x_2}({\bf v}_{1}^{(l-1)})^{(2)}|\leq C\delta(x_1)^{l-3},\\
&| ({\bf f}_{1}^{l-1})^{(2)}(x)|\leq C|\partial_{x_1x_1}({\bf v}_{1}^{(l-1)})^{(2)}|\leq C\delta(x_1)^{(2l-5)/2}
\end{split}\quad\text{in}~\Omega_{2R}.
\end{equation}
For the high-order derivatives, we have, for $k\ge 0$,
\begin{equation}\label{xudgjf2}
\begin{split}
&|\partial_{x_1}^k\partial_{x_2}^s({\bf f}_{1}^{l-1})^{(1)}(x)|\leq C\delta(x_1)^{(2l-2s-k-6)/2}~\text{for}~0\le s\le 2l-3,\\
&\partial_{x_2}^s({\bf f}_{1}^{l-1})^{(1)}(x)=0~\text{for}~s\ge 2l-2,\\
& |\partial_{x_1}^k\partial_{x_2}^s({\bf f}_{1}^{l-1})^{(2)}(x)|\le C\delta(x_1)^{(2l-2s-k-5)/2}~\text{for}~0\le s\le 2l-2,\\
&\partial_{x_2}^s({\bf f}_{1}^{l-1})^{(2)}(x)=0~\text{for}~s\ge 2l-1.
\end{split}
\end{equation}
Thus, by using \eqref{estgreen}, \eqref{def_v1l1}, \eqref{xydgjf}, and \eqref{xudgjf2}, we have, in $\Omega_{2R}$,
\begin{equation*}
	|({\bf v}_{1}^{l})^{(1)}(x) | \le C\delta(x_1)^{2}|({\bf f}_{1}^{l-1})^{(1)}(x)|\leq C\delta(x_1)^{l-1},
\end{equation*}
and for $k\ge 0$, $2\le s\le 2l-1$, 
\begin{equation}\label{need1}
\begin{split}
&|\partial_{x_1}^k\partial_{x_2}({\bf v}_{1}^{l})^{(1)}(x) | \le C\delta(x_1)|\partial_{x_1}^{k}\partial_{x_2}(G(y,x_2){\bf f}_{1}^{l-1})^{(1)}(x)|\leq C\delta(x_1)^{(2l-k-4)/2},\\
&|\partial_{x_1}^k\partial_{x_2}^s({\bf v}_{1}^{l})^{(1)}(x) | \le C|\partial_{x_1}^k\partial_{x_2}^{s-2}({\bf f}_{1}^{l-1})^{(1)}(x)|\leq C\delta(x_1)^{(2l-2s-k-2)/2},
\end{split}
\end{equation}
while, for $s\ge 2l$, by \eqref{fzl1}, and \eqref{xudgjf2},
\begin{equation*}
\partial_{x_2}^s({\bf v}_{1}^{l})^{(1)}(x) =- \frac{1}{\mu}\partial_{x_2}^{s-2}({\bf f}_{1}^{l-1})^{(1)}(x)=0.
\end{equation*}
Similarly, by using \eqref{estgreen}, \eqref{def_v1l2} and \eqref{xydgjf}--\eqref{need1}, we derive, in $\Omega_{2R}$,
\begin{equation*}
|({\bf v}_{1}^{l})^{(2)}(x) | \leq C\delta(x_1)^{2}|({\bf f}_{1}^{l-1})^{(2)}|+C\delta(x_1)^{2}|\partial_{x_1x_2}({\bf v}_{1}^{l})^{(1)}|\leq C\delta(x_1)^{(2l-1)/2},
\end{equation*}
and for $k\ge 0$, 
\begin{equation*}
|\partial_{x_1}^k\partial_{x_2}^s({\bf v}_{1}^{l})^{(2)}(x) | \le C\delta(x_1)^{(2l-2s-k-1)/2}~\text{for}~0\le s\le 2l,~\partial_{x_2}^s({\bf v}_{1}^{l})^{(2)}(x)=0~\text{for}~s\ge 2l+1.
\end{equation*}
Thus, \eqref{xydgj2d1} and \eqref{xydgj2d2} hold for $j\ge 2$.
 
Consequently, 
\begin{equation*}
\begin{split}
&|({\bf f}_1^m)^{(1)}(x)|\leq C|\partial_{x_1x_1}({\bf v}_1^{m})^{(1)}|+C|\partial_{x_1x_2}({\bf v}_1^{m})^{(2)}|\leq C\delta(x_1)^{m-2},\\
&|({\bf f}_1^m)^{(2)}(x)|\leq C|\partial_{x_1x_1}({\bf v}_1^{m})^{(2)}|\leq C\delta(x_1)^{m-3/2}
\end{split}\quad\text{in}~\Omega_{2R}.
\end{equation*}
Thus, 
\begin{equation*}
	|{\bf f}_{1}^{m}(x)|\leq C\delta(x_1)^{m-2}\quad\text{in}~\Omega_{2R},
\end{equation*}
and similarly, for $1\le s\le m-1$,
\begin{equation*}
|\nabla^s{\bf f}_{1}^{m}(x)|\leq C\delta(x_1)^{m-2-s}\quad\text{in}~\Omega_{2R}.
\end{equation*}
By using Proposition \ref{propnew}, we obtain, for $x\in\Omega_{R}$,
\begin{equation}\label{zx2dv1l}
\Big\|\nabla^{m}\Big({\bf u}_{11}-\sum_{l=1}^{m}{\bf v}_{1}^{l}\Big)\Big\|_{L^\infty(\Omega_{\delta(x_1)/2}(x_1))}\leq C.
\end{equation}

By virtue of  \eqref{xydgj2d2}, we have
\begin{equation}\label{zx2dv1l1}
|\nabla^m{\bf v}_{1}^{l}(x)|\leq C\delta(x_1)^{-\frac{m+1}{2}}~\text{for}~l\le \frac m2,\quad|\nabla^m{\bf v}_{1}^{l}(x)|\leq C\delta(x_1)^{l-m-1}~\text{for}~m\ge l\ge\frac{m+1}{2}.
\end{equation}
Hence, it follows from \eqref{zx2dv1l} and \eqref{zx2dv1l1} that
\begin{equation*}
|\nabla^{m}{\bf u}_{11}(x) | \leq \Big|\nabla^{m}\sum_{l=1}^{m}{\bf v}_{1}^{l}(x)\Big| +C\leq \sum_{l=1}^{m}\Big|\nabla^{m}{\bf v}_{1}^{l}(x)\Big| +C\leq C\delta(x_1)^{-\frac{m+1}{2}}\quad\mbox{for}~x\in \Omega_R.
\end{equation*}
The proof of Theorem \ref{thm3.1} is completed.
\end{proof}

\subsection{Estimates of $\nabla^{m}{\bf u}_{13}$}
Using the same process, we obtain the estimate of $\nabla^{m}{\bf u}_{13}$ as follows.

\begin{theorem}\label{thm3.3}
Under the same assumption as in Theorem \ref{mainth}, let ${\bf u}_{13}$ be the solution to \eqref{equ_v1} for $\alpha=3$. Then for sufficiently small $0<\varepsilon<1/2$ and for $m\ge 1$, we have 
\begin{align*}
|\nabla^{m}{\bf u}_{13}(x) | \le C\delta(x_1)^{-m/2}\quad\text{for}~x\in\Omega_{R}.
\end{align*}
\end{theorem}

\begin{proof}[Proof of Theorem \ref{thm3.3}]
	Denote
	\begin{equation*}
		{\bf v}_{3}^{1}(x):=\frac{x_2+\varepsilon/2+h_2(x_1)}{\delta(x_1)}{\boldsymbol\psi}_{3}=\frac{x_2+\varepsilon/2+h_2(x_1)}{\delta(x_1)}\begin{pmatrix}
			x_{2}\\
			-x_{1}
		\end{pmatrix} ~\text{for}~x\in\Omega_{2R}.
	\end{equation*}
A simple calculation gives 
\begin{align*}
|\partial_{x_1}({\bf v}_{3}^{1})^{(1)}|\leq&\, C\delta(x_1)^{1/2},~\quad\quad |\partial_{x_2}({\bf v}_{3}^{1})^{(1)}|\leq C, \\
|\partial_{x_1}({\bf v}_{3}^{1})^{(2)}|\leq&\, C,\quad\quad\quad\quad\quad\quad |\partial_{x_2}({\bf v}_{3}^{1})^{(2)}|\leq C\delta(x_1)^{-1/2},
\end{align*}
and
\begin{align}\label{v31dds}
\begin{split}
&|\partial_{x_1x_1}({\bf v}_{3}^{1})^{(1)}|\leq C,\quad
|\partial_{x_1x_2}({\bf v}_{3}^{1})^{(1)}|\leq C\delta(x_1)^{-1/2}, \quad|\partial_{x_2x_2}
({\bf v}_{3}^{1})^{(1)}|\leq  C\delta(x_1)^{-1}, \\ 
&|\partial_{x_1x_1}({\bf v}_{3}^{1})^{(2)}|\leq C\delta(x_1)^{-1/2}, \quad|\partial_{x_1x_2}({\bf v}_{3}^{1})^{(2)}|\leq C\delta(x_1)^{-1}, \quad\partial_{x_2x_2}({\bf v}_{3}^{1})^{(2)}=0.
\end{split}
\end{align}
In addition, for the high-order derivatives, we have, for $k\ge 2$ and $0\le s\le 2$, in $\Omega_{2R}$
\begin{equation*}
\begin{split}
&|\partial_{x_1}^k\partial_{x_{2}}^s({\bf v}_{3}^{1})^{(1)}|\leq C\delta(x_1)^{(2-2s-k)/2},\quad \partial_{x_{2}}^3({\bf v}_{3}^{1})^{(1)}=0,\\
&|\partial_{x_1}^k({\bf v}_{3}^{1})^{(2)}|\leq C\delta(x_1)^{-(k-1)/2},\quad\quad\quad|\partial_{x_1}^k\partial_{x_{2}}({\bf v}_{3}^{1})^{(2)}|\leq C\delta(x_1)^{-(k+1)/2}.
\end{split}
\end{equation*}

Denote 
$${\bf f}_{3}^{1}(x):=\mathcal{L}_{\lambda,\mu}{\bf v}_{3}^{1}(x).$$ We have
\begin{equation*}
\begin{split}
&({\bf f}_{3}^{1})^{(1)}=(\lambda+2\mu)\partial_{x_1x_1}({\bf v}_{3}^{1})^{(1)}+\mu\partial_{x_2x_2}({\bf v}_{3}^{1})^{(1)}+(\lambda+\mu)\partial_{x_1x_2}({\bf v}_{3}^{1})^{(2)},\\
&({\bf f}_{3}^{1})^{(2)}=\mu\partial_{x_1x_1}({\bf v}_{3}^{1})^{(2)}+(\lambda+\mu)\partial_{x_1x_2}({\bf v}_{3}^{1})^{(1)}.
\end{split}
\end{equation*}
By \eqref{v31dds}, it is easy to check that the leading term in $({\bf f}_{3}^{1})^{(1)}$ is $\mu\partial_{x_2x_2}({\bf v}_{3}^{1})^{(1)}+(\lambda+\mu)\partial_{x_1x_2}({\bf v}_{3}^{1})^{(2)}$, being of the order $\delta(x_1)^{-1}$, while both terms in $({\bf f}_{3}^{1})^{(2)}$ are of the order $\delta(x_1)^{-1/2}$. Thus,
\begin{align*}
|{\bf f}_{3}^{1}(x)|=|\mathcal{L}_{\lambda,\mu}{\bf v}_{3}^{1}(x)|\leq \big|({\bf f}_{3}^{1})^{(1)}\big|+\big|({\bf f}_{3}^{1})^{(2)}\big|\le C\delta(x_1)^{-1}\quad\text{in}~\Omega_{2R}.
\end{align*}
By applying Proposition \ref{propnew}, it is not difficult to see that $\nabla{\bf v}_{3}^{1}$ captures all singular terms in $\nabla {\bf u}_{13}$. Consequently,
\begin{equation*}
|\nabla {\bf u}_{13}(x)|\leq C|\nabla{\bf v}_{3}^{1}(x)|+C\le C\delta(x_1)^{-1/2},\quad x\in \Omega_{R}.
\end{equation*}

We denote 
$${\bf f}_3^{l}(x):=\sum_{j=1}^{l}\mathcal{L}_{\lambda,\mu}{\bf v}_{3}^{j}(x)= {\bf f}_3^{l-1}(x)+\mathcal{L}_{\lambda,\mu}{\bf v}_{3}^{l}(x).$$
For $l=2$, we choose ${\bf v}_{3}^{2}(x)=\begin{pmatrix}
	({\bf v}_{3}^{2})^{(1)}(x),
	({\bf v}_{3}^{2})^{(2)}(x)
\end{pmatrix}^{T}$ satisfying ${\bf v}_{3}^{2}(x)=0$ on $\Gamma^+_{2R}\cup\Gamma^-_{2R}$ and in $\Omega_{2R}$,
\begin{align*}
	({\bf v}_{3}^{2})^{(1)}(x)&=\frac{1}{\mu}\int_{-\varepsilon/2-h_2(x_1)}^{\varepsilon/2+h_1(x_1)} G(y,x_2) \Big(\mu\partial_{x_2x_2}({\bf v}_{3}^{1})^{(1)}+(\lambda+\mu)\partial_{x_1x_2}({\bf v}_{3}^{1})^{(2)}\Big)(x_1,y) dy,\\
({\bf v}_{3}^{2})^{(2)}(x)&=\frac{1}{\lambda+2\mu}\int_{-\varepsilon/2-h_2(x_1)}^{\varepsilon/2+h_1(x_1)} G(y,x_2) \Big(({\bf f}_{3}^{1})^{(2)}+(\lambda+\mu)\partial_{x_{1}x_2}({\bf v}_{3}^{2})^{(1)}\Big)(x_1,y) dy,
\end{align*}
so that
\begin{align}\label{fzl1-}
	\begin{split}
		\mu\partial_{x_2x_2}({\bf v}_{3}^{2})^{(1)}&=-\mu\partial_{x_2x_2}({\bf v}_{3}^{1})^{(1)}-(\lambda+\mu)\partial_{x_1x_2}({\bf v}_{3}^{1})^{(2)},\\
		(\lambda+2\mu)\partial_{x_2x_2}({\bf v}_{3}^{2})^{(2)}&=-({\bf f}_3^{1})^{(2)}-(\lambda+\mu)\partial_{x_1x_2}({\bf v}_{3}^{2})^{(1)}
	\end{split}\quad\text{in}~\Omega_{2R}.
\end{align}
The purpose of constructing \({\bf v}_3^2\) is to utilize \(\mu \partial_{x_2x_2} ({\bf v}_3^2)^{(1)}\) to cancel the leading terms of \(({\bf f}_3^1)^{(1)}\), and to apply \((\lambda + 2\mu) \partial_{x_2x_2} ({\bf v}_3^2)^{(2)}\) to eliminate \(({\bf f}_3^1)^{(2)} + (\lambda + \mu) \partial_{x_1x_2} ({\bf v}_3^2)^{(1)}\), which is of order \(\delta(x_1)^{-1/2}\). As a result, the magnitude of \({\bf f}_3^2\) becomes smaller than that of \({\bf f}_3^1\). Indeed, a simple calculation yields 
\begin{align*}
	|({\bf v}_{3}^{2})^{(1)}|\leq C\delta(x_1),\quad 	|({\bf v}_{3}^{2})^{(2)}|\leq C\delta(x_1)^{3/2},
\end{align*}
and for $k\ge 0$, in $\Omega_{2R}$,
\begin{align}\label{gs}
	\begin{split}
	&|\partial_{x_1}^k\partial_{x_2}^s({\bf v}_{3}^{2})^{(1)}(x) | \le C\delta(x_1)^{\frac{2-2s-k}{2}}~\text{for}~0\le s\le 2,\\
	&\partial_{x_2}^s({\bf v}_{3}^{2})^{(1)}(x)=0~\text{for}~ s\ge 3,\\
	&|\partial_{x_1}^k\partial_{x_2}^s({\bf v}_{3}^{2})^{(2)}(x) | \le C\delta(x_1)^{\frac{3-2s-k}{2}}~\text{for}~0\le s\le 3,\\
	&\partial_{x_2}^s({\bf v}_{3}^{2})^{(2)}(x)=0~\text{for}~ s\ge 4.
\end{split}
\end{align}
Then, in view of \eqref{fzl1-}, we have
\begin{align*}
	({\bf f}_3^2)^{(1)}&=({\bf f}_3^1)^{(1)}+\mu\partial_{x_2x_2}({\bf v}_{3}^{2})^{(1)}+(\lambda+2\mu)\partial_{x_1x_1}({\bf v}_{3}^{2})^{(1)}+(\lambda+\mu)\partial_{x_1x_2}({\bf v}_{3}^{2})^{(2)}\\
	&=(\lambda+2\mu)\partial_{x_1x_1}({\bf v}_{3}^{1})^{(1)}+(\lambda+2\mu)\partial_{x_1x_1}({\bf v}_{3}^{2})^{(1)}+(\lambda+\mu)\partial_{x_1x_2}({\bf v}_{3}^{2})^{(2)},\\
	({\bf f}_3^2)^{(2)}&=({\bf f}_3^1)^{(2)}+(\lambda+\mu)\partial_{x_{1}x_2}({\bf v}_{3}^{2})^{(1)}+(\lambda+2\mu)\partial_{x_{2}x_2}({\bf v}_{3}^{2})^{(2)}+\mu\partial_{x_{1}x_1}({\bf v}_{3}^{2})^{(2)}\\
	&=\mu\partial_{x_{1}x_1}({\bf v}_{3}^{2})^{(2)}.
\end{align*}
Thus, it follows from \eqref{v31dds} and \eqref{gs} that, in $\Omega_{2R}$,
\begin{align*}
|({\bf f}_3^2)^{(1)}(x)|\leq C, \quad |({\bf f}_3^2)^{(2)}(x)|\leq C\delta(x_1)^{1/2},
\end{align*}
and
$$|\nabla({\bf f}_3^2)(x)|\leq C\delta(x_1)^{-1}.$$
By applying Proposition \ref{propnew}, we have, for $x\in \Omega_{R}$,
 $$\Big\|\nabla^{2}\Big({\bf u}_{13}-({\bf v}_{3}^{1}+{\bf v}_{3}^{2})\Big)\Big\|_{L^\infty(\Omega_{\delta(x_1)/2}(x_1))}\leq C,$$
and thus,
\begin{equation*}
	|\nabla^2 {\bf u}_{13}(x)|\leq C|\nabla^2({\bf v}_{3}^{1}+{\bf v}_{3}^{2})(x)|+C\le C\delta(x_1)^{-1}.
\end{equation*}

For $l\ge3$, we choose  inductively ${\bf v}_{3}^{l}=\begin{pmatrix}
	({\bf v}_{3}^{l})^{(1)},
	({\bf v}_{3}^{l})^{(2)}
\end{pmatrix}^{T}$ such that ${\bf v}_{3}^{l}(x)=0$ on $\Gamma^+_{2R}\cup\Gamma^-_{2R}$, and in $\Omega_{2R}$,
\begin{align*}
({\bf v}_{3}^{l})^{(1)}(x)&=\frac{1}{\mu}\int_{-\varepsilon/2-h_2(x_1)}^{\varepsilon/2+h_1(x_1)} G(y,x_2) ({\bf f}_{3}^{l-1})^{(1)}(x_1,y) dy,\\
({\bf v}_{3}^{l})^{(2)}(x)&=\frac{1}{\lambda+2\mu}\int_{-\varepsilon/2-h_2(x_1)}^{\varepsilon/2+h_1(x_1)} G(y,x_2) \Big(({\bf f}_{3}^{(l-1)})^{(2)}+(\lambda+\mu)\partial_{x_{1}x_2}({\bf v}_{3}^{l})^{(1)}\Big)(x_1,y) dy,
\end{align*}
so that, 
\begin{align*}
	\begin{split}
		\mu\partial_{x_2x_2}({\bf v}_{3}^{l})^{(1)}&=-({\bf f}_{1}^{l-1})^{(1)},\\
		(\lambda+2\mu)\partial_{x_2x_2}({\bf v}_{3}^{l})^{(2)}&=-\big({\bf f}_3^{(l-1)}\big)^{(2)}-(\lambda+\mu)\partial_{x_{1}x_2}({\bf v}_{3}^{l})^{(1)}
	\end{split}\quad\text{in}~\Omega_{2R}.
\end{align*}

We can inductively show that, for $l\ge 2$,
\begin{equation*}
	|({\bf v}_{3}^{l})^{(1)}(x) | \le C\delta(x_1)^{l-1},\quad|({\bf v}_{1}^{l})^{(2)}(x) | \le C\delta(x_1)^{(2l-1)/2}\quad\text{in}~\Omega_{2R},
\end{equation*}
and for $k\ge 0$, in $\Omega_{2R}$,
\begin{align*}
&|\partial_{x_1}^k\partial_{x_2}^s({\bf v}_{3}^{l})^{(1)}(x) | \le C\delta(x_1)^{\frac{2l-2s-k-2}{2}}~\text{for}~0\le s\le 2l-2,\\
&\partial_{x_2}^s({\bf v}_{3}^{l})^{(1)}(x)=0~\text{for}~ s\ge 2l-1,\\
&|\partial_{x_1}^k\partial_{x_2}^s({\bf v}_{3}^{l})^{(2)}(x) | \le C\delta(x_1)^{\frac{2l-2s-k-1}{2}}~\text{for}~0\le s\le 2l-1,\\
&\partial_{x_2}^s({\bf v}_{3}^{l})^{(2)}(x)=0~\text{for}~ s\ge 2l.
\end{align*}
The proof is the same as that of Theorem \ref{thm3.1}, so we omit it here. In addition, a straightforward calculation leads to, in $\Omega_{2R}$,
\begin{align}\label{f1l32d}
|{\bf f}_{3}^{m}(x)| \le C \delta(x_1)^{m-2}, \quad |\nabla^s{\bf f}_{3}^{m}(x)|
\le C \delta(x_1)^{m-2-s},~1\le s\le m-1,
\end{align}
and 
\begin{align}\label{zxv3l2d}
|\nabla^m{\bf v}_{3}^{l}(x)|\leq C\delta(x_1)^{-m/2}~\text{for}~l\le \frac{m+1}{2},~|\nabla^m{\bf v}_{3}^{l}(x)|\leq C\delta(x_1)^{l-m-1}~\text{for}~\frac{m+2}{2}\leq l\le m.
\end{align}
By using \eqref{f1l32d} and Proposition \ref{propnew}, we obtain, for $x\in\Omega_{R}$,
\begin{equation}\label{zyxz1}
	\Big\|\nabla^{m}\Big({\bf u}_{13}-\sum_{l=1}^{m}{\bf v}_{3}^{l}\Big)\Big\|_{L^\infty(\Omega_{\delta(x_1)/2}(x_1))}\leq C.
\end{equation}
As before, by \eqref{zxv3l2d} and \eqref{zyxz1}, Theorem \ref{thm3.3} follows immediately.
\end{proof}


\section{Proof of Theorems \ref{mainth}}\label{sec3}
With Theorems \ref{thm3.1} and \ref{thm3.3} at hand, we are in a position to complete the proof of Theorems \ref{mainth}. As in \cite{bll1,bll2}, we decompose the solution to \eqref{maineqn} as follows
\begin{equation}\label{decom_u}
{\bf u}(x)=\sum_{i=1}^{2}\sum_{\alpha=1}^{3}C_i^{\alpha}{\bf u}_{i\alpha}(x)+{\bf u}_{0}(x),\quad x\in\,\Omega ,
\end{equation}
where ${\bf u}_{i\alpha},{\bf u}_{0}\in{C}^{2}(\Omega;\mathbb R^d)$, respectively, satisfying
\begin{equation}\label{equuia}
\begin{cases}
\mathcal{L}_{\lambda,\mu}{\bf u}_{i\alpha}=0&\mathrm{in}~\Omega,\\
{\bf u}_{i\alpha}=\boldsymbol{\psi}_{\alpha}&\mathrm{on}~\partial{D}_{i},\\
{\bf u}_{i\alpha}=0&\mathrm{on}~\partial{D_{j}}\cup\partial{D},~j\neq i,
\end{cases}
\quad i=1,2,~\alpha=1,2,3,
\end{equation}
and
\begin{equation}\label{equu0}
\begin{cases}
\mathcal{L}_{\lambda,\mu}{\bf u}_{0}=0&\mathrm{in}~\Omega,\\
{\bf u}_{0}=0&\mathrm{on}~\partial{D}_{1}\cup\partial{D_{2}},\\
{\bf u}_{0}=\boldsymbol{\varphi}&\mathrm{on}~\partial{D}.
\end{cases}
\end{equation}
We write
\begin{equation}\label{nablau_dec}
\nabla^{m}{\bf u}=\sum_{\alpha=1}^{3}\left(C_{1}^{\alpha}-C_{2}^{\alpha}\right)\nabla^{m}{\bf u}_{1\alpha}+\nabla^{m} {\bf u}_{b}\quad\mbox{in}~\Omega,
\end{equation}
where ${\bf u}_{b}:=\sum_{\alpha=1}^{3}C_{2}^{\alpha}({\bf u}_{1\alpha}+{\bf u}_{2\alpha})+{\bf u}_{0}$. Since ${\bf u}_{1\alpha}+{\bf u}_{2\alpha}-{\boldsymbol\psi}_\alpha=0$ on $\partial D_1\cup\partial D_2$, $\alpha=1,2,3$, it was proved in \cite[Theorem 1.1]{llby} that, for $m\ge1$, $\nabla^{m}\big({\bf u}_{1\alpha}+{\bf u}_{2\alpha}-{\boldsymbol\psi}_\alpha\big)$ has no singularity in the narrow region. More precisely, we have
\begin{theorem}\label{thm31}
Let ${\bf u}_{i\alpha}$ and  ${\bf u}_{0}$ be the solution to \eqref{equuia} and \eqref{equu0}, $i=1,2$, $\alpha=1,2$. Then we have, for $m\ge 1$, 
\begin{align*}
|\nabla^m({\bf u}_{1\alpha}+{\bf u}_{2\alpha})(x)|,~|\nabla^m{\bf u}_{0}(x)|\leq Ce^{-\frac{C}{\sqrt{\delta(x_1)}}}\quad \text{for}~x\in\Omega_{R},
\end{align*}
and
\begin{align*}
|\nabla({\bf u}_{13}+{\bf u}_{23})(x)|\leq1+Ce^{-\frac{C}{\sqrt{\delta(x_1)}}} ,\quad |\nabla^{m+1}({\bf u}_{13}+{\bf u}_{23})(x)|\leq Ce^{-\frac{C}{\sqrt{\delta(x_1)}}} \quad \text{for}~x\in\Omega_{R}.
\end{align*}
\end{theorem}

We have the following result, which is from \cite[Lemma 4.1, Proposition 4.2]{bll1}.
\begin{lemma}\label{cdgj}
Let $C_i^\alpha$ be defined in \eqref{decom_u}. Then
\begin{equation*}
|C_i^\alpha|\leq C, \quad i=1,2, \quad\alpha=1,2,3,
\end{equation*}
and 
\begin{equation*}
|C_1^\alpha-C_2^\alpha|\leq C\sqrt{\varepsilon},\quad d=2,\quad \alpha=1,2.
\end{equation*}
\end{lemma}

\begin{proof}[Proof of Theorem \ref{mainth}]
By virtue of \eqref{nablau_dec}, Theorems \ref{thm3.1},  \ref{thm3.3}, \ref{thm31}, and Lemma \ref{cdgj}, we have, for $x\in\Omega_{R}$,
\begin{align*}
|\nabla^{m}{\bf u}(x)|\leq&\, |C_1^1-C_2^1||\nabla^{m} {\bf u}_{11}(x)|+|C_1^2-C_2^2||\nabla^{m} {\bf u}_{12}(x)|+C|\nabla^{m} {\bf u}_{13}(x)|+Ce^{-\frac{C}{\sqrt{\delta(x_1)}}}+C\\
\leq&\, C\sqrt{\varepsilon}\delta(x_1)^{-\frac{m+1}{2}}+C\delta(x_1)^{-m/2}+Ce^{-\frac{C}{\sqrt{\delta(x_1)}}}+C\\
\leq&\, C\delta(x_1)^{-m/2}.
\end{align*}
This completes the proof of Theorem \ref{mainth}.
\end{proof}


\section{New Ingredients for higher dimensional cases and Proofs of the Main Estimates}\label{sec4}

This section is devoted to proving Theorem \ref{mainth3d} by an analogous argument as in Section \ref{sec3}.  We present the main differences and a sketch of the proof of Theorems \ref{mainth3d}. 

Recalling that a basis of $\Psi$ in $\mathbb{R}^d$ is
$$\left\{\boldsymbol{e}_{i},~x_{j}\boldsymbol{e}_{k}-x_{k}\boldsymbol{e}_{j}:~1\leq\,i\leq\,d,~1\leq\,j<k\leq\,d\right\},$$
where $\boldsymbol{e}_{1},\cdots,\boldsymbol{e}_{d}$ denote the standard basis of $\mathbb{R}^{d}$. We denote this basis as $\{\boldsymbol{\psi}_{\alpha}\}$, $\alpha=1,2,\dots,\frac{d(d+1)}{2}$ such that 
\begin{align*}
\boldsymbol{\psi}_{\alpha}&=\boldsymbol{e}_{\alpha}, \quad
\alpha=1,2,\dots,d,\\
 \boldsymbol{\psi}_{\alpha}&=x_{i_\alpha}\boldsymbol{e}_{j_\alpha}-x_{j_\alpha}\boldsymbol{e}_{i_\alpha}, \quad\alpha=d+1,\dots,\frac{d(d-1)}{2}+1,~1\le i_\alpha< j_\alpha\le d-1,\\
 \boldsymbol{\psi}_{\alpha}&=x_d \boldsymbol{e}_{\alpha-\frac{d(d-1)}{2}-1}-x_{\alpha-\frac{d(d-1)}{2}-1} \boldsymbol{e}_{d}, \quad\alpha=\frac{d(d-1)}{2}+2,\dots,\frac{d(d+1)}{2}.
 \end{align*}

We consider the following Dirichlet problem
\begin{equation}\label{equ_v13d}
\begin{cases}
\mathcal{L}_{\lambda,\mu}{\bf u}_{1\alpha}=0&\mathrm{in}~\Omega,\\
{\bf u}_{1\alpha}=\boldsymbol{\psi}_{\alpha}&\mathrm{on}~\partial{D}_{1},\\
{\bf u}_{1\alpha}=0&\mathrm{on}~\partial{D_{2}}\cup\partial{D}.
\end{cases}
\end{equation}
For each ${\bf u}_{1\alpha}(x)$, $\alpha=1,2,\dots,\frac{d(d+1)}{2}$, we construct a family of auxiliary functions to capture the singularity term of $\nabla^{m}{\bf u}_{1\alpha}(x)$ for $m\ge 1$. First,

\subsection{Estimates of $\nabla^{m}{\bf u}_{1\alpha}$, $\alpha=1,2,\dots,d$} 
We have  
\begin{theorem}\label{thm4.1}
Under the same assumption as in Theorem \ref{mainth3d}, let ${\bf u}_{1\alpha}$ be the solution to \eqref{equ_v13d} for $\alpha=1,2,\dots,d$. Then for sufficiently small $0<\varepsilon<1/2$ and for $m\ge 1$, we have
 \begin{align*}
|\nabla^m {\bf u}_{1\alpha}(x)|\leq C\delta(x')^{-\frac{m+1}{2}}\quad\text{for}~x\in\Omega_{R}.
 \end{align*}
\end{theorem}
This is a significant improvement of previous results in \cite{bll2,LX23}. Denote the vertical distance between $D_{1}$ and $D_{2}$ by 
$$\delta(x'):=\varepsilon+h_{1}(x')+h_{2}(x').$$

\begin{proof}
We only present the proof of the case when $\alpha=1$, since the other cases are analogous. 

{\bf Step 1.} 
We first construct ${\bf v}_{1}^{1}=\begin{pmatrix}
({\bf v}_{1}^{1})^{(1)},0,\dots,0,
({\bf v}_{1}^{1})^{(d)}
\end{pmatrix}^{T}$ such that ${\bf v}_{1}^{1}(x)={\bf u}_{11}(x)$ on $\Gamma^+_{2R}\cup\Gamma^-_{2R}$, and
\begin{align}
({\bf v}_{1}^{1})^{(1)}(x)&=\frac{x_d+\varepsilon/2+h_2(x')}{\delta(x')}{\boldsymbol\psi}_{1} \quad\text{in}~\Omega_{2R},\label{defv113d}\\
({\bf v}_{1}^{1})^{(d)}(x)&=\frac{\lambda+\mu}{\lambda+2\mu}\int_{-\varepsilon/2-h_2(x')}^{\varepsilon/2+h_1(x')} G(y,x_d)\partial_{x_1x_d} ({\bf v}_{1}^{1})^{(1)}(x',y) dy\quad\text{in}~\Omega_{2R},\label{defv1123d}
\end{align}
where 
\begin{align}\label{defgreenfunction3d}
	G(y,x_d)=\frac{1}{\delta(x')}\begin{cases}
		(\frac{\varepsilon}{2}+h_1(x')-x_d)(y+\frac{\varepsilon}{2}+h_2(x')), \quad -\frac{\varepsilon}{2}-h_2(x')\le y\le x_d,\\
	(\frac{\varepsilon}{2}+h_1(x')-y)(x_d+\frac{\varepsilon}{2}+h_2(x')), \quad x_d\le x\le \frac{\varepsilon}{2}+h_1(x')
	\end{cases}
\end{align}
is the Green function. By \eqref{defv1123d} and \eqref{defgreenfunction3d}, we have
\begin{equation}\label{3dfz1}
	(\lambda+2\mu)\partial_{x_dx_d}({\bf v}_{1}^{1})^{(d)}=-(\lambda+\mu)\partial_{x_1x_d} ({\bf v}_{1}^{1})^{(1)}\quad\text{in}~\Omega_{2R}.
\end{equation}
By \eqref{h1h14} and \eqref{defgreenfunction3d}, we have
\begin{equation}\label{3dgredgj}
\begin{split}
&|G(y,x_d)|\leq C\delta(x'), \quad\quad\quad\quad\quad |\nabla_{x'}G(y,x_d)|\leq C\delta(x')^{1/2},\quad\quad|\partial_{x_d}G(y,x_d)|\leq C,\\
&|\nabla_{x'}^kG(y,x_d)|\leq C\delta(x')^{-(k-2)/2},\quad1_{y\neq x_d}\cdot\partial_{x_d}^kG(y,x_d)=0\quad\mbox{for}~2\le k\le m+1.
\end{split}
\end{equation}
By virtue of  \eqref{h1h14} and \eqref{defv113d}, a direct calculation gives, for $1\le i< d$, in $\Omega_{2R}$,
 \begin{equation}\label{3dgjdsgjv111}
	\begin{split}
&|\nabla_{x'}({\bf v}_{1}^{1})^{(1)}|\leq C\delta(x')^{-1/2},\quad\,\quad|\partial_{x_d}({\bf v}_{1}^{1})^{(1)}|\leq C\delta(x')^{-1},\\
&|\nabla_{x'}\partial_{x_d}({\bf v}_{1}^{1})^{(1)}|\leq C\delta(x')^{-3/2},\quad\partial_{x_dx_d}({\bf v}_{1}^{1})^{(1)}=0,\\
&|\nabla_{x'}^k({\bf v}_{1}^{1})^{(1)}|\leq C\delta(x')^{-k/2},\quad\quad |\nabla_{x'}^{k-1}\partial_{x_{d}}({\bf v}_{1}^{1})^{(1)}|\leq C\delta(x')^{-(k+1)/2}~\text{for}~2\le k\le m+1.
\end{split}
\end{equation}
In addition, by \eqref{defv1123d} and \eqref{3dgredgj}, we have
\begin{equation*}
|({\bf v}_{1}^{1})^{(d)}|\leq C\delta(x')^{1/2}, \quad |\partial_{x_{d}}({\bf v}_{1}^{1})^{(d)}|\leq C\delta(x')|\partial_{x_1x_d} ({\bf v}_{1}^{1})^{(1)}|\leq C\delta(x')^{-1/2}\quad\text{in}~\Omega_{2R},
\end{equation*}
and
\begin{equation*}
 |\nabla_{x'}({\bf v}_{1}^{1})^{(d)}|\leq C\delta(x')^2|\nabla_{x'}\partial_{x_1}\partial_{x_d} ({\bf v}_{1}^{1})^{(1)}|+C\delta(x')^{3/2}|\partial_{x_1}\partial_{x_d} ({\bf v}_{1}^{1})^{(1)}|\leq C\quad\text{in}~\Omega_{2R}.
\end{equation*}
Similarly, for the high-order derivatives, by \eqref{defv1123d}, \eqref{defgreenfunction3d}, \eqref{3dgredgj} and \eqref{3dgjdsgjv111}, we obtain, for $2\le k\le m+1$,
\begin{equation}\label{dsgj13dv112}
	\begin{split}
	 &|\nabla_{x'}^k({\bf v}_{1}^{1})^{(d)}|\leq C\delta(x')^{-(k-1)/2}, \\
	 &|\nabla_{x'}^{k-1}\partial_{x_{d}}({\bf v}_{1}^{1})^{(d)}|\leq C\delta(x')^{-k/2},\\
&|\nabla_{x'}^{k-2}\partial_{x_{d}}^2({\bf v}_{1}^{1})^{(d)}|\leq C|\nabla_{x'}^{k-1}\partial_{x_d} ({\bf v}_{1}^{1})^{(1)}| \leq C\delta(x')^{-(k+1)/2}
\end{split}\quad\text{in}~\Omega_{2R}.
\end{equation}

Denote 
$${\bf f}_{1}^{1}(x):=\mathcal{L}_{\lambda,\mu}{\bf v}_{1}^{1}(x).$$ Then, by using \eqref{3dfz1}, we have
\begin{align*}
({\bf f}_{1}^{1})^{(1)}&=\mu\Delta_{x'}({\bf v}_{1}^{1})^{(1)}+(\lambda+\mu)\partial_{x_{1}x_1}({\bf v}_{1}^{1})^{(1)}+(\lambda+\mu)\partial_{x_1x_d}({\bf v}_{1}^{1})^{(d)},\\
({\bf f}_{1}^{1})^{(j)}&=(\lambda+\mu)\Big(\partial_{x_1x_j}({\bf v}_{1}^{1})^{(1)}+\partial_{x_jx_d}({\bf v}_{1}^{1})^{(d)}\Big),\quad\,1<j<d,\\
({\bf f}_{1}^{1})^{(d)}&=(\lambda+\mu)\partial_{x_1x_d}({\bf v}_{1}^{1})^{(1)}+ (\lambda+2\mu)\partial_{x_dx_d}({\bf v}_{1}^{1})^{(d)}+\mu \Delta_{x'}({\bf v}_{1}^{1})^{(d)}=\mu \Delta_{x'}({\bf v}_{1}^{1})^{(d)}.
\end{align*}
In view of \eqref{3dgjdsgjv111} and \eqref{dsgj13dv112}, we obtain
\begin{equation*}
|({\bf f}_{1}^{1})^{(j)}(x)|\leq C\delta(x')^{-1},\,\,1\le j<d,\quad |({\bf f}_{1}^{1})^{(d)}(x)|\leq C\delta(x')^{-1/2}\quad\text{in}~\Omega_{2R}.
\end{equation*}
Hence,
\begin{equation}\label{m=13d}
	|{\bf f}_{1}^{1}(x)|\leq \sum_{j=1}^{d}|({\bf f}_{1}^{1})^{(j)}(x)|\leq C \delta(x')^{-1}.
\end{equation}
By Proposition \ref{propnew}, we have Theorem \ref{thm4.1} holds true for $m=1$ by using estimate \eqref{m=13d}, 
\begin{align*}
	\|\nabla\big({\bf u}_{11}-{\bf v}_{1}^{1}\big)\|_{L^\infty(\Omega_{\delta(x')/2}(x'))}\leq C.
\end{align*}
Thus, we have
\begin{equation*}
	|\nabla {\bf u}_{11}(x)|\leq |\nabla {\bf v}_{1}^{1}(x)|+C\leq C\delta(x')^{-1}.
\end{equation*}

{\bf Step 2.} For $l\ge2$, we denote 
$${\bf f}_1^{l}(x):=\sum_{j=1}^{l}\mathcal{L}_{\lambda,\mu}{\bf v}_{1}^{j}(x)= {\bf f}_1^{l-1}(x)+\mathcal{L}_{\lambda,\mu}{\bf v}_{1}^{l}(x).$$ 
We choose ${\bf v}_{1}^{l}(x)=\begin{pmatrix}
({\bf v}_{1}^{l})^{(1)}(x),({\bf v}_{1}^{l})^{(2)}(x),\dots,
({\bf v}_{1}^{l})^{(d)}(x)
\end{pmatrix}^{T}$ satisfying ${\bf v}_{1}^{l}(x)=0$ on $\Gamma^+_{2R}\cup\Gamma^-_{2R}$, and in $\Omega_{2R}$,
\begin{align*}
	({\bf v}_{1}^{l})^{(i)}(x)&=\frac{1}{\mu}\int_{-\varepsilon/2-h_2(x')}^{\varepsilon/2+h_1(x')} G(y,x_d) ({\bf f}_{1}^{l-1})^{(i)}(x',y) dy,\quad 1\leq i\leq d-1,\\
	({\bf v}_{1}^{l})^{(d)}(x)&=\frac{1}{\lambda+2\mu}\int_{-\varepsilon/2-h_2(x')}^{\varepsilon/2+h_1(x')} G(y,x_d) \Big(({\bf f}_{1}^{l-1})^{(d)}+(\lambda+\mu)\sum_{i=1}^{d-1}\partial_{x_ix_d}({\bf v}_1^l)^{(i)}\Big)(x',y) dy.
\end{align*}
Then, we have, for $1\le i<d$, 
\begin{align*}
	\begin{split}
\mu\partial_{x_dx_d}({\bf v}_{1}^{l})^{(i)}&=-({\bf f}_{1}^{l-1})^{(i)},\\
(\lambda+2\mu)\partial_{x_dx_d}({\bf v}_{1}^{l})^{(d)}&=-({\bf f}_{1}^{l-1})^{(d)}-(\lambda+\mu)\sum_{i=1}^{d-1}\partial_{x_ix_d}({\bf v}_1^l)^{(i)}
	\end{split}\quad\text{in}~\Omega_{2R}.
\end{align*}
Then, for $1\le i\le d-1$, we obtain
\begin{align}\label{3df1}
	({\bf f}_1^l)^{(i)}&=({\bf f}_1^{l-1})^{(i)}+\mu \Delta_{x} ({\bf v}_{1}^{l})^{(i)}+(\lambda+\mu)\partial_{x_{i}}(\nabla\cdot{\bf v}_{1}^{l})\nonumber\\
	&=\mu\Delta_{x'} ({\bf v}_{1}^{l})^{(i)}+(\lambda+\mu)\partial_{x_{i}}(\nabla\cdot{\bf v}_{1}^{l}),
\end{align}
and
\begin{equation}\label{3df2}
	({\bf f}_1^l)^{(d)}=({\bf f}_1^{l-1})^{(d)}+\mu\Delta({\bf v}_{1}^{l})^{(d)}+(\lambda+\mu)\partial_{x_{d}}(\nabla\cdot{\bf v}_{1}^{l})=\mu\Delta_{x'}({\bf v}_{1}^{l})^{(d)}.
\end{equation}

As in the proof of Theorem \ref{thm3.1}, we can inductively obtain the following estimates in $\Omega_{2R}$, for any $l\ge 1$ and $1\le i\le d-1$,
\begin{equation*}
	|({\bf v}_{1}^{l})^{(i)}(x) | \le C\delta(x')^{l-1},\quad|({\bf v}_{1}^{l})^{(d)}(x) | \le C\delta(x')^{(2l-1)/2},
\end{equation*}
and for $k\ge 0$, 
\begin{align}\label{xydgj3d2}
	\begin{split}
	&|\nabla_{x'}^k\partial_{x_d}^s({\bf v}_{1}^{l})^{(i)}(x) | \le C\delta(x')^{\frac{2l-2s-k-2}{2}}~
\text{for}~s\le 2l-1,\\
&\partial_{x_d}^s({\bf v}_{1}^{l})^{i}(x)=0~\text{for}~s\ge2l,\\
&|\nabla_{x'}^k\partial_{x_d}^s({\bf v}_{1}^{l})^{d}(x) | \le C\delta(x')^{\frac{2l-2s-k-1}{2}}~\text{for}~s\le 2l,\\
&\partial_{x_d}^s({\bf v}_{1}^{l})^{(d)}(x)=0~\text{for}~s\ge2l+1.
\end{split}
\end{align}
Combining with \eqref{3df1}--\eqref{xydgj3d2}, we have
\begin{equation}\label{1f1l13d}
	|{\bf f}_{1}^{m}(x)|=\Big|\sum_{l=1}^{m}\mathcal{L}_{\lambda,\mu}{\bf v}_{1}^{l}(x)\Big|\leq C\delta(x')^{m-2},
\end{equation}
and similarly, 
\begin{equation}\label{1f1l13d1}
|\nabla^s{\bf f}_{1}^{m}(x)|\leq C\delta(x')^{m-2-s}, \quad 1\le s\le m-1.
\end{equation}
By using \eqref{1f1l13d}, \eqref{1f1l13d1}, and Proposition \ref{propnew}, we obtain, for $x\in\Omega_{R}$,
$$\Big\|\nabla^{m}\Big({\bf u}_{11}-\sum_{l=1}^{m}{\bf v}_{1}^{l}\Big)\Big\|_{L^\infty(\Omega_{\delta(x')/2}(x'))}\leq C.$$

By virtue of  \eqref{xydgj3d2}, we have
\begin{equation*}
|\nabla^m{\bf v}_{1}^{l}(x)|\leq C\delta(x')^{-\frac{m+1}{2}}~\text{for}~l\le \frac m2,\quad	|\nabla^m{\bf v}_{1}^{l}(x)|\leq C\delta(x')^{l-m-1}~\text{for}~l\ge \frac{m+1}{2}.
\end{equation*}
Thus, 
\begin{equation*}
	|\nabla^{m}{\bf u}_{11}(x) | \leq \Big|\sum_{l=1}^{m}\nabla^{m}{\bf v}_{1}^{l}(x)\Big| +C\leq C\delta(x')^{-\frac{m+1}{2}}\quad\mbox{for}~x\in \Omega_R.
\end{equation*}
Therefore, Theorem \ref{thm4.1} is proved.
\end{proof}

\subsection{Estimates of $\nabla^{m}{\bf u}_{1\alpha}$, $\alpha=d+1,\dots,\frac{d(d+1)}{2}$}
\begin{theorem}\label{thm4v33d}
Under the same assumption as in Theorem \ref{mainth3d}, let ${\bf u}_{1\alpha}$ be the solution to \eqref{equ_v13d} for $\alpha=d+1,\dots,\frac{d(d+1)}{2}$. Then for sufficiently small $0<\varepsilon<1/2$ and for $m\ge 1$, we have
\begin{align*}
|\nabla^{m}{\bf u}_{1\alpha}(x)| \le C\delta(x')^{-m/2}\quad\text{for}~x\in\Omega_{R}.
\end{align*}
\end{theorem}
\begin{proof}
We only prove the case of $\boldsymbol{\psi}_{d+1}=(x_2,-x_1,0,\dots,0)^{\mathrm{T}}$, since other cases are similar.   

{\bf Step 1.} We define
\begin{equation*}
	{\bf v}_{\alpha}^{1}(x)=\frac{x_d+\varepsilon/2+h_2(x')}{\delta(x')}{\boldsymbol\psi}_{\alpha} \quad\text{for}~x\in\Omega_{2R}.
\end{equation*}
By direct calculations, we have
\begin{equation}\label{gjds13d}
\begin{split}
 &|\nabla_{x'}^k({\bf v}_{\alpha}^{1})^{(1)}|, |\nabla_{x'}^k({\bf v}_{\alpha}^{1})^{(2)}|\leq C\delta(x')^{-(k-1)/2},\quad\quad\quad\quad 0\le k\le m+1,\\
 & |\nabla_{x'}^{k-1}\partial_{x_{d}}({\bf v}_{\alpha}^{1})^{(1)}|, |\nabla_{x'}^{k-1}\partial_{x_{d}}({\bf v}_{\alpha}^{1})^{(2)}|\leq C\delta(x')^{-k/2},\quad 1\le k\le m+1,
 \end{split}
\end{equation}
and 
\begin{equation}\label{gjds13d1}
\partial_{x_{d}x_d}({\bf v}_{\alpha}^{1})^{(1)}=\partial_{x_{d}x_d}({\bf v}_{\alpha}^{1})^{(2)}=0.
\end{equation}

Denote ${\bf f}_{\alpha}^{1}(x):=\mathcal{L}_{\lambda,\mu}{\bf v}_{\alpha}^{1}(x)$, so that
\begin{equation*}
({\bf f}_{\alpha}^{1})^{(1)}=\mu \Delta({\bf v}_{\alpha}^{1})^{(1)}+(\lambda+\mu)\partial_{x_{1}}(\nabla\cdot{\bf v}_{\alpha}^{1}),\quad ({\bf f}_{\alpha}^{1})^{(2)}=\mu \Delta({\bf v}_{\alpha}^{1})^{(2)}+(\lambda+\mu)\partial_{x_{2}}(\nabla\cdot{\bf v}_{\alpha}^{1}),
\end{equation*}
and for $j\ge3$,
\begin{equation*}
({\bf f}_{\alpha}^{1})^{(j)}=(\lambda+\mu)\partial_{x_{j}}(\nabla\cdot{\bf v}_{\alpha}^{1}).
\end{equation*}
It follows from \eqref{gjds13d} and \eqref{gjds13d1} that, for $1\le i<d$,
\begin{equation*}
|({\bf f}_{\alpha}^{1})^{(i)}(x)|\leq C\delta(x')^{-1/2},\quad |({\bf f}_{\alpha}^{1})^{(d)}(x)|\leq C\delta(x')^{-1}.
\end{equation*}
Thus,
\begin{equation}\label{f113d1}
|{\bf f}_{\alpha}^{1}(x)|\leq C\sum_{i=1}^{d}|({\bf f}_{\alpha}^{1})^{(i)}(x)|\leq C\delta(x')^{-1}.
\end{equation}

By using \eqref{f113d1} and Proposition \ref{propnew}, we have
\begin{align*}
	\|\nabla\big({\bf u}_{\alpha}-({\bf v}_{\alpha}^{11}+{\bf v}_{\alpha}^{12})\big)\|_{L^\infty(\Omega_{\delta(x')/2}(x'))}\leq C.
\end{align*}
By virtue of  \eqref{gjds13d} for $k=0$, we conclude that Theorem \ref{thm4v33d} holds true for $m=1$, that is,
\begin{equation*}
	|\nabla {\bf u}_{\alpha}(x)|\leq |\nabla {\bf v}_{\alpha}^{1}(x)|+C\leq C\delta(x')^{-1/2}.
\end{equation*}

{\bf Step 2.} For $l\ge2$, denote 
$${\bf f}_{\alpha}^{l}(x):=\sum_{j=1}^{l}\mathcal{L}_{\lambda,\mu}{\bf v}_{\alpha}^{j}(x)= {\bf f}_{\alpha}^{l-1}(x)+\mathcal{L}_{\lambda,\mu}{\bf v}_{\alpha}^{l}(x).$$ 
We choose ${\bf v}_{\alpha}^{l}$ satisfying ${\bf v}_{\alpha}^{l}(x)=0$ on $\Gamma^+_{2R}\cup\Gamma^-_{2R}$, and 
\begin{equation*}
	({\bf v}_{\alpha}^{l})^{(d)}(x)=\frac{1}{\lambda+2\mu}\int_{-\varepsilon/2-h_2(x')}^{\varepsilon/2+h_1(x')} G(y,x_d) ({\bf f}_{\alpha}^{l-1})^{(d)}(x',y) dy\quad\text{in}~\Omega_{2R},
\end{equation*}
and for $1\le i\le d-1$,
\begin{equation*}
	({\bf v}_{_{\alpha}}^{l})^{(i)}(x)=\frac{1}{\mu}\int_{-\varepsilon/2-h_2(x')}^{\varepsilon/2+h_1(x')} G(y,x_d) \Big(({\bf f}_{\alpha}^{l-1})^{(i)}+(\lambda+\mu)\partial_{x_{i}x_d}({\bf v}_{\alpha}^{l})^{(d)}\Big)(x',y) dy\quad\text{in}~\Omega_{2R}.
\end{equation*}
Then, we have
\begin{equation*}
(\lambda+2\mu)\partial_{x_dx_d}({\bf v}_{\alpha}^{l})^{(d)}=-({\bf f}_{\alpha}^{l-1})^{(d)}\quad\text{in}~\Omega_{2R},
\end{equation*}
and for $1\le i\le d-1$,
\begin{equation*}
	\mu\partial_{x_dx_d}({\bf v}_{1}^{l})^{(i)}=-({\bf f}_{\alpha}^{l-1})^{(i)}-(\lambda+\mu)\partial_{x_{i}x_d}({\bf v}_{\alpha}^{l})^{(d)}\quad\text{in}~\Omega_{2R}.
\end{equation*}
Thus, for $1\le i\le d-1$,
\begin{equation}\label{f3d1k}
	\begin{split}
({\bf f}_{\alpha}^{l})^{(d)}&= ({\bf f}_{\alpha}^{l-1})^{(d)}+(\lambda+2\mu)\partial_{x_{d}x_d}({\bf v}_{\alpha}^{l})^{(d)}+(\lambda+\mu)\sum_{i=1}^{d-1}\partial_{x_ix_d}({\bf v}_{\alpha}^{l})^{(i)}\\
&=(\lambda+\mu)\sum_{i=1}^{d-1}\partial_{x_ix_d}({\bf v}_{\alpha}^{l})^{(i)},\\
({\bf f}_{\alpha}^{l})^{(i)}&=({\bf f}_{\alpha}^{l-1})^{(i)}+\mu\Delta ({\bf v}_{\alpha}^{l})^{(i)}+(\lambda+\mu)\partial_{x_{i}}(\nabla\cdot {\bf v}_{\alpha}^{l})\\
&=\mu\Delta_{x'} ({\bf v}_{\alpha}^{l})^{(i)}+(\lambda+\mu)\sum_{j=1}^{d-1}\partial_{x_ix_j}({\bf v}_{\alpha}^{l})^{(i)}.
\end{split}
\end{equation}

As in the proof of Theorem \ref{thm3.1}, by induction we have, for any $1\le l\le m$ and $1\le i\le d-1$,
\begin{equation*}
	|({\bf v}_{\alpha}^{l})^{(i)}(x) | \le C\delta(x')^{(2l-1)/2},\quad|({\bf v}_{\alpha}^{l})^{(d)}(x) | \le C\delta(x')^{l-1}\quad\text{in}~\Omega_{2R},
\end{equation*}
and for $k\ge 0$, in $\Omega_{2R}$,
\begin{align}\label{gj1}
	\begin{split}
	&|\nabla_{x'}^k\partial_{x_d}^s({\bf v}_{\alpha}^{l})^{(i)}(x) | \le C\delta(x')^{\frac{2l-2s-k-1}{2}}~\text{for}~0\le s\le 2l-1,\\
	&\partial_{x_d}^s({\bf v}_{\alpha}^{l})^{(i)}(x)=0~\text{for}~ s\ge 2l,\\
	&|\nabla_{x'}^k\partial_{x_d}^s({\bf v}_{\alpha}^{l})^{(d)}(x) | \le C\delta(x')^{\frac{2l-2s-k-2}{2}}~\text{for}~0\le s\le 2l-2,\\
	&\partial_{x_d}^s({\bf v}_{\alpha}^{l})^{(d)}(x)=0~\text{for}~ s\ge 2l-1.
\end{split}
\end{align}
Thus, combining \eqref{gj1} with \eqref{f3d1k} yields
\begin{equation}\label{f3dgj111}
	|{\bf f}_{\alpha}^{m}(x)|\leq C\delta(x')^{m-2}\quad\text{in}~\Omega_{2R},
\end{equation}
and similarly, for $s\le m-1$,
\begin{equation}\label{f3dgj2}
|\nabla^s{\bf f}_{\alpha}^{m}(x)|\leq C\delta(x')^{m-2-s}\quad\text{in}~\Omega_{2R}.
\end{equation}
By using \eqref{f3dgj111}, \eqref{f3dgj2}, and Proposition \ref{propnew}, we obtain, for $x\in\Omega_{R}$,
$$\Big\|\nabla^{m}\Big({\bf u}_{\alpha}-\sum_{l=1}^{m}{\bf v}_{\alpha}^{l}\Big)\Big\|_{L^\infty(\Omega_{\delta(x')/2}(x'))}\leq C.$$

By virtue of \eqref{gj1}, 
\begin{align*}
|\nabla^m{\bf v}_{\alpha}^{l}(x)|\leq C\delta(x')^{-m/2}~\text{for}~l\le \frac{m+1}{2},\quad 	|\nabla^m{\bf v}_{\alpha}^{l}(x)|\leq C\delta(x')^{l-m-1}~\text{for}~l\ge \frac{m+2}{2}.
\end{align*}
Thus, 
\begin{equation*}
	|\nabla^{m}{\bf u}_{\alpha}(x) | \leq \Big|\sum_{l=1}^{m}\nabla^{m}{\bf v}_{\alpha}^{l}(x)\Big| +C\leq C\delta(x')^{-m/2}\quad\mbox{for}~x\in \Omega_R.
\end{equation*}
Therefore, Theorem \ref{thm4v33d} is proved.
\end{proof}

\subsection{The Completion of the Proof of Theorem \ref{mainth3d}}
As in Section \ref{sec3}, we decompose the solution to \eqref{maineqn} as follows
\begin{equation}\label{decom_u3d}
{\bf u}(x)=\sum_{i=1}^{2}\sum_{\alpha=1}^{d(d+1)/2}C_i^{\alpha}{\bf u}_{i\alpha}(x)+{\bf u}_{0}(x),\quad x\in\,\Omega,
\end{equation}
where ${\bf u}_{i\alpha},{\bf u}_{0}\in{C}^{2}(\Omega;\mathbb R^d)$, respectively, satisfying
\begin{equation}\label{equuia3d}
\begin{cases}
\mathcal{L}_{\lambda,\mu}{\bf u}_{i\alpha}=0&\mathrm{in}~\Omega,\\
{\bf u}_{i\alpha}=\boldsymbol{\psi}_{\alpha}&\mathrm{on}~\partial{D}_{i},\\
{\bf u}_{i\alpha}=0&\mathrm{on}~\partial{D_{j}}\cup\partial{D},~j\neq i,
\end{cases}
\quad i=1,2,~\alpha=1,2,\dots,\frac{d(d+1)}{2},
\end{equation}
and
\begin{equation}\label{equu03d}
\begin{cases}
\mathcal{L}_{\lambda,\mu}{\bf u}_{0}=0&\mathrm{in}~\Omega,\\
{\bf u}_{0}=0&\mathrm{on}~\partial{D}_{1}\cup\partial{D_{2}},\\
{\bf u}_{0}=\boldsymbol{\varphi}&\mathrm{on}~\partial{D}.
\end{cases}
\end{equation}
We write
\begin{equation}\label{nablau_dec3d}
\nabla^{m}{\bf u}(x)=\sum_{\alpha=1}^{d(d+1)/2}\left(C_{1}^{\alpha}-C_{2}^{\alpha}\right)\nabla^{m}{\bf u}_{1\alpha}(x)+\nabla^{m} {\bf u}_{b}(x)\quad\mbox{in}~\Omega,
\end{equation}
where ${\bf u}_{b}:=\sum_{\alpha=1}^{d(d+1)/2}C_{2}^{\alpha}({\bf u}_{1\alpha}+{\bf u}_{2\alpha})+{\bf u}_{0}$. 

It is obvious that ${\bf u}_{1\alpha}+{\bf u}_{2\alpha}-\boldsymbol{\psi}_\alpha=0$ on $\partial D_1\cup \partial D_2$. We have the following results, which is from \cite[Theorem 1.1]{llby}.
\begin{theorem}
Let ${\bf u}_{i\alpha}$ and  ${\bf u}_{0}$ be the solution to \eqref{equuia3d} and \eqref{equu03d}, $i=1,2$, $\alpha=1,2,\dots,d$. Then we have, for $m\ge 1$,
\begin{align}\label{ubgj3d}
|\nabla^m({\bf u}_{1\alpha}+{\bf u}_{2\alpha})(x)|,\quad\quad|\nabla^m{\bf u}_{0}(x)|\leq Ce^{-\frac{C}{\sqrt{\delta(x')}}} \quad \text{for}~x\in\Omega_{R},
\end{align}
and, for $\alpha=d+1,d+2,\dots,\frac{d(d+1)}{2}$,
\begin{align}\label{ubgj3d1}
|\nabla({\bf u}_{1\alpha}+{\bf u}_{2\alpha})(x)|\leq 1+ Ce^{-\frac{C}{\sqrt{\delta(x')}}},\quad|\nabla^{m+1}({\bf u}_{1\alpha}+{\bf u}_{2\alpha})(x)|\leq Ce^{-\frac{C}{\sqrt{\delta(x')}}} \quad \text{for}~x\in\Omega_{R}.
\end{align}
\end{theorem}
For these constants $C_i^\alpha$, the following result is from \cite[Proposition 4.1]{bll2}.
\begin{lemma}\label{cdgj3d}
Let $C_i^\alpha$ be defined in \eqref{decom_u}. Then
\begin{align*}
|C_i^\alpha|\leq C, \quad i=1,2; \quad\alpha=1,2\dots,\frac{d(d+1)}{2},
\end{align*}
and for $d=3$,
\begin{align*}
|C_1^\alpha-C_2^\alpha|\leq \frac{C}{|\log\varepsilon|}, \quad\alpha=1,2,3.
\end{align*}
\end{lemma}
We are now in a position to prove Theorem \ref{mainth3d}.
\begin{proof}[Proof of Theorem \ref{mainth3d}]
For $d=3$, by virtue of \eqref{nablau_dec3d}--\eqref{ubgj3d1}, Lemma \ref{cdgj3d} and Theorems \ref{thm4.1} and \ref{thm4v33d}, we have, for $x\in\Omega_{R}$,
\begin{align*}
|\nabla^{m}{\bf u}(x)|\leq&\, \sum_{\alpha=1}^3|C_1^\alpha-C_2^\alpha||\nabla^{m} {\bf u}_{1\alpha}(x)|+\sum_{\alpha=4}^6C|\nabla^{m} {\bf u}_{1\alpha}(x)|+C+ Ce^{-\frac{C}{\sqrt{\delta(x')}}}\\
\leq&\, C|\log\varepsilon|^{-1}\delta(x')^{-\frac{m+1}{2}}+C\delta(x')^{-m/2}.
\end{align*}
Similarly, for $d\ge 4$, by using \eqref{nablau_dec3d}--\eqref{ubgj3d1}, Lemma \ref{cdgj3d} and Theorems \ref{thm4.1} and \ref{thm4v33d} again, we have, for $x\in\Omega_{R}$,
\begin{align*}
	|\nabla^{m}{\bf u}(x)|\leq C \sum_{\alpha=1}^{d(d+1)/2}|\nabla^{m} {\bf u}_{1\alpha}(x)|+C+ Ce^{-\frac{C}{\sqrt{\delta(x')}}}\leq C\delta(x')^{-\frac{m+1}{2}}.
\end{align*}
This completes the proof of Theorem \ref{mainth3d}.
\end{proof}


\section{Asymptotics and lower bounds in dimensions two and three}\label{sec5}

In this section, under certain symmetry assumption on the domain and boundary data, we derive an explicit asymptotic formula for the higher derivatives in dimensions two and three. To this end, we need to consider the limiting case when $\varepsilon=0$. Let ${\bf u}^*$ be the solution to
\begin{align*}
	\begin{cases}
		\mathcal{L}_{\lambda, \mu}{\bf u}^{*}=0,\quad&\hbox{in}\ \Omega^{*},\\
		{\bf u}^{*}=\sum_{\alpha=1}^{d(d+1)/2}C_{*}^{\alpha}\boldsymbol{\psi}_{\alpha},&\hbox{on}\ \partial D_{1}^{*}\cup\partial D_{2}^{*},\\
		\int_{\partial{D}_{1}^{*}}\frac{\partial{{\bf u}}^{*}}{\partial\nu}\big|_{+}\cdot\boldsymbol{\psi}_{\beta}+\int_{\partial{D}_{2}^{*}}\frac{\partial{{\bf u}}^{*}}{\partial\nu}\big|_{+}\cdot\boldsymbol{\psi}_{\beta}=0,&\beta=1,\cdots,\frac{d(d+1)}{2},\\
		{\bf u}^{*}=\boldsymbol{\varphi},&\hbox{on}\ \partial{D},
	\end{cases}
\end{align*}
where $D_{1}^{*}:=\{x\in\mathbb R^{d}:~ x+(0',\frac{\varepsilon}{2})\in D_{1}\}$, $D_{2}^{*}:=\{x\in\mathbb R^{d}:~ x-(0',\frac{\varepsilon}{2})\in D_{2}\}$, and $\Omega^{*}:=D\setminus\overline{D_{1}^{*}\cup D_{2}^{*}}$, and the constants $C_*^\alpha$ are uniquely determined by the solution ${\bf u}^*$. We define linear functionals of $\boldsymbol{\varphi}$:
\begin{align*}
	b_{1\beta}^{*}[\boldsymbol{\varphi}]:=\int_{\partial{D}_{1}^{*}}\frac{\partial {\bf u}^{*}}{\partial \nu}\Big|_{+}\cdot\boldsymbol{\psi}_{\beta},\quad\mbox{for}~\beta=1,\cdots,d(d+1)/2.
\end{align*}

We assume $D_1\cup D_2$ and $D$ are symmetric with respect to each axis $x_i$ and $\boldsymbol{\varphi}(x)=-\boldsymbol{\varphi}(-x)$ for $x\in\partial D$. It was shown in \cite[Proposition 5.4]{LX23} that in this case,
 \begin{equation}\label{cxiangdeng}
C_1^\alpha=C_2^\alpha, \quad \text{for}~\alpha=d+1,\dots,d(d+1)/2,~d=2,3,
 \end{equation}
so that we only need to focus on the estimates of the components ${\bf u}_{i\alpha}$, $\alpha=1,2,\dots,d$, $i=1,2$. We further assume $h_1$ and $h_2$ are quadratic and symmetric with respect to the plane $\{x_d=0\}$, say, $h_1(x')=h_2(x')=\frac{1}{2}|x'|^2$ for $|x'|\leq2R$. Then, the vertical distance between $D_{1}$ and $D_{2}$ becomes $\delta(x')=\varepsilon+|x'|^2$.  

We will construct a series of explicit polynomial functions ${\bf v}_{\alpha}^{l}=\begin{pmatrix}
({\bf v}_{\alpha}^{l})^{(1)},
({\bf v}_{\alpha}^{l})^{(2)}
\end{pmatrix}^{T}$ in dimension two and ${\bf v}_{\alpha}^{l}=\begin{pmatrix}
({\bf v}_{\alpha}^{l})^{(1)},
({\bf v}_{\alpha}^{l})^{(2)},
({\bf v}_{\alpha}^{l})^{(3)}
\end{pmatrix}^{T}$ in dimension three, which satisfy the conditions ${\bf v}_{\alpha}^{1}(x)={\bf u}_{1\alpha}$ on $\Gamma^{\pm}_{2R}$ and ${\bf v}_{\alpha}^{l}(x)=0$ on $\Gamma^{\pm}_{2R}$, $l\ge 2$, in order to approximate the solutions ${\bf u}_{1\alpha}$ for $\alpha=1,2,\dots,d$ in the narrow region. We treat \eqref{fzl1} as a second-order ordinary differential equations with respect to $x_2$. By using the boundary conditions ${\bf v}_{1}^{l}(x)=0$, $l\ge 2$ on $\Gamma^{\pm}_{2R}$, we can explicitly derive the auxiliary functions ${\bf v}_1^l$, as presented in \eqref{def_v1l1} and \eqref{def_v1l2}. In this section, under the above symmetry assumptions, we employ a direct method to obtain the solutions to \eqref{fzl1}. 

\subsection{Dimension Two}

For ${\bf u}_{11}$, let 
	\begin{align}\label{def2dv11i}
(	{\bf v}_{1}^{1})^{(1)}(x)=\frac{x_{2}}{\varepsilon+x_1^2}+\frac{1}{2},\quad({\bf v}_{1}^{1})^{(2)}(x)=\mathcal{P}_{12,1}(x_1)\Big(x_2^2-\frac{(\varepsilon+x_1^2)^2}{4}\Big)\quad\mbox{in}~\Omega_{2R}.
\end{align}

A direct calculation yields, in $\Omega_{2R}$, 
\begin{align}\label{5.1}
\partial_{x_1x_2} ({\bf v}_{1}^{1})^{(1)} =-\frac{2x_1}{(\varepsilon+x_1^2)^2},\quad \partial_{x_2x_2} ({\bf v}_{1}^{1})^{(2)}= 2\mathcal{P}_{12,1}(x_1).
\end{align}
By means of \eqref{def_v112gx}, we have $\mathcal{P}_{12,1}(x_1)=\frac{\lambda+\mu}{\lambda+2\mu}\frac{x_1}{(\varepsilon+x_1^2)^2}$. Then, it is easy to verify that $|\mathcal{L}_{\lambda,\mu}{\bf v}_{1}^{1}|\leq C\delta(x_1)^{-1}$, which is less than $|\mathcal{L}_{\lambda,\mu}(({\bf v}_{1}^{1})^{(1)}{\boldsymbol\psi}_{1})(x)|$. This result is consistent with that in Section \ref{sec2}. Recalling that $\delta(x_1)=\varepsilon+x_1^2$, the Green function \eqref{defgreenfunction} becomes
\begin{align}\label{xingreen}
	G(y, x_{2})=\frac{1}{\varepsilon+x_1^2}\begin{cases}
		(\frac{\varepsilon+x_1^2}{2}-x_{2})(y+\frac{\varepsilon+x_1^2}{2}), \quad -\frac{\varepsilon+x_1^2}{2}\le y\le x_{2},\\
		(\frac{\varepsilon+x_1^2}{2}-y)(x_{2}+\frac{\varepsilon+x_1^2}{2}), \quad x_{2}\le y\le \frac{\varepsilon+x_1^2}{2}.
	\end{cases}
\end{align}
By virtue of \eqref{def_v112}, \eqref{5.1} and \eqref{xingreen}, we have, in $\Omega_{2R}$,
\begin{align*}
	({\bf v}_{1}^{1})^{(2)}(x)&=-\frac{\lambda+\mu}{\lambda+2\mu}\frac{2x_1}{(\varepsilon+x_1^2)^2}\int_{-\frac12(\varepsilon+x_1^2)}^{\frac12(\varepsilon+x_1^2)} G(y,x_2)dy\\
	&=\frac{\lambda+\mu}{\lambda+2\mu}\frac{x_1}{(\varepsilon+x_1^2)^2}\Big(x_2^2-\frac{(\varepsilon+x_1^2)^2}{4}\Big).
\end{align*}
For $l\ge 2$, we choose ${\bf v}_{1}^{l}(x)$ satisfying ${\bf v}_{1}^{l}(x)=0$ on $\partial\Omega$, with the following form in $\Omega_{2R}$,
\begin{align}\label{def2dvi1i}
\begin{split}
	({\bf v}_{1}^{l})^{(1)}(x) &= \sum_{i=1}^{l-1} \mathcal{P}_{l1,i}(x_1)x_2^{2l-2i-1}\Big(x_2^2-\frac{(\varepsilon+x_1^2)^2}{4} \Big),\\
	({\bf  v}_{1}^{l})^{(2)}(x)&= \sum_{i=1}^{l}\mathcal{P}_{l2,i}(x_1)x_2^{2l-2i}\Big(x_2^2-\frac{(\varepsilon+x_1^2)^2}{4}\Big).
\end{split}
\end{align}
By using \eqref{fzl1}, we have
 $\mathcal{P}_{21,1}(x_1)=\frac{2\lambda+3\mu}{3(\lambda+2\mu)}\frac{\varepsilon-3x_1^2}{(\varepsilon+x_1^2)^3}$, $\mathcal{P}_{l1,i}(x_1)$, $0\leq i\le l-2$, satisfy the following recursive formula
\begin{align}\label{dtgs1}
	\mathcal{P}_{l1,i+1}(x_1)&= \frac{(\varepsilon+x_1^2)^2}{4}\mathcal{P}_{l1,i}(x_1)\nonumber\\
	&-\frac{1}{ (a_{l,i}-1)a_{l,i}} \left( \frac{\lambda+\mu}{\mu}(a_{l,i}-1)
	\tilde{\mathcal{P}}'_{(l-1)2,i+1}(x_1)+\frac{\lambda+2\mu}{\mu}\tilde{\mathcal{P}}''_{(l-1)1,i+1}(x_1)\right),
\end{align}
where $a_{l,i}=2(l-i)-1$, and $\mathcal{P}_{l2,i}(x_1)$, $0\leq i\le l-1$, satisfy
\begin{align}\label{dtgs2}
\mathcal{P}_{l2,i+1}(x_1)=&\,\frac{(\varepsilon+x_1^2)^2}{4}\mathcal{P}_{l2,i}(x_1)\nonumber\\
&-\frac{1}{a_{l,i}(a_{l,i}+1)}\left(\frac{\lambda+\mu}{\lambda+2\mu}a_{l,i}\tilde{\mathcal{P}}'_{l1,i+1}(x_1)+\frac{\mu}{\lambda+2\mu}\tilde{\mathcal{P}}''_{(l-1)2,i+1}(x_1)\right),
\end{align}
where 
$$\tilde{\mathcal{P}}_{l1,i}= \mathcal{P}_{l1,i}(x_1)-\frac{(\varepsilon+x_1^2)^2}{4}\mathcal{P}_{l1,i-1}(x_1), \quad\tilde{\mathcal{P}}_{l2,i}(x_1)= \mathcal{P}_{l2,i}(x_1)-\frac{(\varepsilon+x_1^2)^2}{4}\mathcal{P}_{l2,i-1}(x_1).$$ Here we use the convention that $\mathcal{P}_{l1,i}(x_1)=0$ if $i\notin\{1,2,\dots,l-1\}$ and $\mathcal{P}_{l2,i}(x_1)=0$ if  $i\notin\{1,2,\dots,l\}$. 
Then, we have
\begin{equation*}
\nabla^{m}{\bf u}_{11}(x)=\nabla^{m}\sum_{l=1}^{m}{\bf v}_{1}^{l}(x) +O(1)\quad \mbox{for}~
		x\in\Omega_{R}.
\end{equation*} 
This asymptotic formula of $\nabla^{m}{\bf u}_{1\alpha}$ is an essential improvement of previous results in \cite{bll1,LX23}. 

For ${\bf u}_{12}$, similarly, let
	\begin{equation*}
	({\bf v}_{2}^{1})^{(2)}(x)=\frac{x_2}{\varepsilon+x_1^2}+\frac{1}{2},\quad\quad({\bf v}_{2}^{1})^{(1)}(x)= \mathcal{P}_{12,1}(x_1) \Big(x_2^2-\frac{(\varepsilon+x_1^2)^2}{4}\Big),
	\end{equation*}
	where $\mathcal{P}_{12,1} (x_1)=\frac{\lambda+\mu}{\mu}\frac{x_1}{(\varepsilon+x_1^2)^2}$ and, for $l\ge 2$,
\begin{equation*}
\begin{split}
&({\bf  v}_{2}^{l})^{(2)}(x)= \sum_{i=1}^{l-1}\mathcal{P}_{l1,i}(x_1)x_2^{2l-2i-1}\Big(x_2^2-\frac{(\varepsilon+x_1^2)^2}{4}\Big),\\
& ({\bf  v}_{2}^{l})^{(1)}(x)= \sum_{i=1}^{l}\mathcal{P}_{l2,i}(x_1)x_2^{2l-2i} \Big( x_2^2-\frac{(\varepsilon+x_1^2)^2}{4}\Big),
\end{split}
\end{equation*}
where $ \mathcal{P}_{21,1}
	(x_1)=\frac{\lambda}{3\mu} \frac{3x_1^2-\varepsilon}{(\varepsilon+x_1^2)^3}$, $\mathcal{P}_{l1,i}(x_1)$ satisfies the recursive formula \eqref{dtgs1} with the factors $\frac{\lambda+\mu}{\lambda+2\mu}$ and $\frac{\mu}{\lambda+2\mu}$ in place of $\frac{\lambda+\mu}{\mu}$ and $\frac{\lambda+2\mu}{\mu}$, respectively, while $\mathcal{P}_{l2,i}(x_1)$ satisfies the recursive formula \eqref{dtgs2} with the factors $\frac{\lambda+\mu}{\mu}$ and $\frac{\lambda+2\mu}{\mu}$ in place of $\frac{\lambda+\mu}{\lambda+2\mu}$ and $\frac{\mu}{\lambda+2\mu}$, respectively. 
Then
\begin{equation*}
\nabla^{m}{\bf u}_{12}(x)=\nabla^{m}\sum_{l=1}^{m}{\bf v}_{2}^{l}(x) +O(1)\quad \mbox{for}~
		x\in\Omega_{R}.
\end{equation*}

For the coefficients $C_i^\alpha$, $i,\alpha=1,2$, it was proven in \cite[Proposition 3.7]{LX23} that 
\begin{lemma}\label{SExsjjzk}
Let $D_1,D_2\subset D$, $\boldsymbol{\varphi}$ be as above and let $C_i^\alpha$ be defined in \eqref{decom_u} in dimension two. The following assertions hold
\begin{align*}
		&C_{1}^{1}-C_{2}^{1}
		=\frac{1}{\pi\mu}\cdot b_{11}^{*}[\boldsymbol{\varphi}]\sqrt{\varepsilon}\left(1+O(\varepsilon^{1/4})\right),\\
		&C_{1}^{2}-C_{2}^{2}
		=\frac{1}{\pi(\lambda+2\mu)}\cdot b_{12}^{*}[\boldsymbol{\varphi}]\sqrt{\varepsilon}\left(1+O(\varepsilon^{1/4})\right).
	\end{align*}
\end{lemma}

Then, by using Lemma \ref{SExsjjzk} and \eqref{cxiangdeng}, we have
\begin{theorem}\label{mainth-5.3}
Let $D_1,D_2\subset D$, $\boldsymbol{\varphi}$ be defined as above and let ${\bf u}\in H^{1}(D;\mathbb R^{2})\cap C^{m}(\bar{\Omega};\mathbb R^{2})$ be a solution to \eqref{maineqn}  in dimension two. Then for $m\ge1$ and sufficiently small $0<\varepsilon<1/4$, we have the following assertions, for $x\in\Omega_{R}$,
	\begin{align*}
		\nabla^{m}{\bf u}(x)=\frac{\sqrt{\varepsilon}}{\pi}\left(\frac{b_{11}^{*}[\boldsymbol{\varphi}]}{\mu}\nabla^{m} \sum_{l=1}^{m}{\bf v}_{1}^{l}(x)+\frac{b_{12}^{*}[\boldsymbol{\varphi}]}{\lambda+2\mu}\nabla^{m} \sum_{l=1}^{m}{\bf v}_{2}^{l}(x)\right)(1+O(\varepsilon^{1/4}))+O(1),
	\end{align*}
	where $\tilde{\bf v}_{\alpha}^{l}$, $\alpha=1,2$, are specific functions defined as above. Especially, they satisfy
	$$\frac{1}{C(\varepsilon+x_1^2)^{(m+1)/2}}\leq|\nabla^{m} \sum_{l=1}^{m}\tilde{\bf v}_{\alpha}^{l}(x)|\leq\frac{C}{(\varepsilon+x_1^2)^{(m+1)/2}},\quad\mbox{for}~ m\geq1.
	$$
\end{theorem}

As a direct consequence, we obtain the asymptotic expansion of the components with the highest singularity in the derivative matrix of any order. For instance, 

\begin{corollary}\label{cor1}
	Under the assumptions of Theorem \ref{mainth-5.3},  let ${\bf u}\in H^{1}(D;\mathbb R^{2})\cap C^{m}(\bar{\Omega};\mathbb R^{2})$ be a solution to \eqref{maineqn}  in dimension two. For $x\in\Omega_{R}$, we have
	
	(i) for the gradient matrix,
	\begin{align*}
		\partial_{x_2} {\bf u}^{(1)}(x)&=\frac{1}{\pi\mu}\cdot b_{11}^{*}[\boldsymbol{\varphi}]\frac{\sqrt{\varepsilon}}{\varepsilon+x_1^2}\left(1+O(\varepsilon^{1/4})\right)+O(1),\\
		\partial_{x_2} {\bf u}^{(2)}(x)&=\frac{1}{\pi(\lambda+2\mu)}\cdot b_{12}^{*}[\boldsymbol{\varphi}]\frac{\sqrt{\varepsilon}}{\varepsilon+x_1^2}\left(1+O(\varepsilon^{1/4})\right)+O(1),
	\end{align*}
	while the other terms are bounded by a constant $C$;
	
	(ii) for the second-order derivatives matrix,
	\begin{align*}
		\partial_{x_2x_2}{\bf u}^{(1)}(x)&=\frac{2(\lambda+\mu)}{\pi\mu(\lambda+2\mu)}\cdot b_{12}^{*}[\boldsymbol{\varphi}]\frac{\sqrt{\varepsilon}x_1}{(\varepsilon+x_1^2)^2}\left(1+O(\varepsilon^{1/4})\right)+O(1)(\varepsilon+x_1^2)^{-1/2},\\
		\partial_{x_1x_2}{\bf u}^{(1)}(x)&=-\frac{2}{\pi\mu}\cdot b_{11}^{*}[\boldsymbol{\varphi}]\frac{\sqrt{\varepsilon}x_1}{(\varepsilon+x_1^2)^2}\left(1+O(\varepsilon^{1/4})\right)+O(1)(\varepsilon+x_1^2)^{-1/2},\\
		\partial_{x_1x_2}{\bf u}^{(2)}(x)&=-\frac{2}{\pi(\lambda+2\mu)}\cdot b_{12}^{*}[\boldsymbol{\varphi}]\frac{\sqrt{\varepsilon}x_1}{(\varepsilon+x_1^2)^2}\left(1+O(\varepsilon^{1/4})\right)+O(1)(\varepsilon+x_1^2)^{-1/2},\\
		\partial_{x_2x_2}{\bf u}^{(2)}(x)&=\frac{2(\lambda+\mu)}{\pi\mu(\lambda+2\mu)}\cdot b_{11}^{*}[\boldsymbol{\varphi}]\frac{\sqrt{\varepsilon}x_1}{(\varepsilon+x_1^2)^2}\left(1+O(\varepsilon^{1/4})\right)+O(1)(\varepsilon+x_1^2)^{-1/2},
	\end{align*}
	while the other terms are bounded by $\frac{C}{\sqrt{\varepsilon+x_1^2}}$.
\end{corollary}

\subsection{Dimension Three}
For $l\ge 2$, $0\leq i\le l-2$, we define $\mathcal{P}_{l1,i}(x')$ and $\mathcal{Q}_{l1,i}(x')$ by the following recursive formulas:
\begin{align}\label{3ddtgs1}
	&\mathcal{P}_{l1,i+1}(x')=\frac{(\varepsilon+|x'|^2)^2}{4}\mathcal{P}_{l1,i}(x')-\frac{1}{ a_{l,i}(a_{l,i}-1)\mu}\Big((\lambda+\mu)(a_{l,i}-1)\partial_{x_\alpha} \tilde{\mathcal{P}}_{(l-1)2,i+1}(x')\nonumber\\
	&\quad+\mu\partial_{x_\beta x_\beta}\tilde{\mathcal{P}}_{(l-1)1,i+1}(x')+ (\lambda+2\mu)\partial_{x_\alpha x_\alpha} \tilde{\mathcal{P}}_{(l-1)1,i+1}(x') +(\lambda+\mu)\partial_{x_1x_2}
	\tilde{\mathcal{Q}}_{(l-1)1,i+1}(x')\Big),
\end{align}
\begin{align}\label{3ddtgs2}
&\mathcal{Q}_{l1,i+1}(x') = \frac{(\varepsilon+|x'|^2)^2}{4}\mathcal{Q}_{l1,i}(x')-\frac{1}{ a_{l,i}(a_{l,i}-1)\mu} \Big( (\lambda+\mu)(a_{l,i}-1)\partial_{x_\beta} \tilde{\mathcal{P}}_{(l-1)2,i+1}(x')\nonumber\\
&\quad+ \mu\partial_{x_\alpha x_\alpha}\tilde{\mathcal{Q}}_{(l-1)1,i+1}(x')
+(\lambda+2\mu)\partial_{x_\beta x_\beta}\tilde{\mathcal{Q}}_{(l-1)1,i+1}(x')+(\lambda+\mu)
	\partial_{x_1x_2} \tilde{\mathcal{P}}_{(l-1)1,i+1}(x')\Big),
\end{align}
where $\alpha,\beta=1,2$, $\alpha\neq \beta$ and $a_{l,i}=2(l-i)-1$, while for $0\le i\le l-1$,
\begin{align}\label{3ddtgs3}
	\mathcal{P}_{l2,i+1}(x')=& \frac{(\varepsilon+|x'|^2)^2}{4}\mathcal{P}_{l2,i}(x')\nonumber\\
	&-\frac{1}{a_{l,i}(a_{l,i}+1)(\lambda+2\mu)}\Big((\lambda+
	\mu)a_{l,i}\cdot\Big(\partial_{x_\alpha}\tilde{\mathcal{P}}_{l1,i+1}(x')+\partial_{x_\beta}\tilde{\mathcal{Q}}_{l1,i+1}(x')\Big)\nonumber\\
	&\quad\quad\quad\quad\quad\quad\quad\quad\quad\quad\quad\, +\mu\Delta_{x'}\tilde{\mathcal{P}}_{(l-1)2,i+1}(x')\Big), 
\end{align}
where 
$$
\tilde{\mathcal{P}}_{l1,i}(x'):=\mathcal{P}_{l1,i}(x')-\frac{(\varepsilon+|x'|^2)^2}{4}\mathcal{P}_{l1,i-1}(x')
$$ 
and 
$$
\tilde{\mathcal{Q}}_{l1,i}(x'):=\mathcal{Q}_{l1,i}(x')-\frac{(\varepsilon+|x'|^2)^2}{4}\mathcal{Q}_{l1,i-1}(x').
$$ 
Here we use the convention that $\mathcal{Q}_{l1,i}(x')=\mathcal{P}_{l1,i}(x')=0$ if  $i\notin\{1,2,\dots,l-1\}$, and $\mathcal{P}_{l2,i}(x')=0$ if  $i\notin\{1,2,\dots,l\}$.

For $\alpha=1,2$, let 
\begin{align*}
		{\bf v}_{\alpha}^{11}(x)= \Big(\frac{x_3}{\varepsilon+|x'|^2}+\frac{1}{2}\Big){\boldsymbol\psi}_{\alpha}, \quad {\bf v}_{\alpha}^{12}(x)= \mathcal{P}_{12,1}(x')\Big(x_{3}^2-\frac{(\varepsilon+|x'|^2)^2}{4}\Big) {\boldsymbol\psi}_{3}\quad\mbox{in}~\Omega_{2R},
	\end{align*}
	where $\mathcal{P}_{12,1}(x')=\frac{\lambda+\mu}{\lambda+2\mu} \frac{x_\alpha}{(\varepsilon+|x'|^2)^2}$, and for $l\ge 2$, $\beta=1,2$, $\beta\neq \alpha$,
\begin{align*}
\begin{split}
&{\bf v}_{\alpha}^{l1}(x)=\sum_{i=1}^{l-1} \Big(\mathcal{P}_{l1,i}(x'){\boldsymbol\psi}_{\alpha}+\mathcal{Q}_{l1,i}(x'){\boldsymbol\psi}_{\beta}\Big)x_{3}^{2l-1-2i}\Big(x_3^2-\frac{(\varepsilon+|x'|^2)^2}{4}\Big),\\
&{\bf v}_{\alpha}^{l2}(x)= \sum_{i=1}^{l}\mathcal{P}_{l2,i}(x')x_{3}^{2l-2i}\Big(x_3^2-\frac{(\varepsilon+|x'|^2)^2}{4}\Big){\boldsymbol\psi}_{3}
\end{split}\quad\mbox{in}~\Omega_{2R},
	\end{align*}
where 
\begin{align*}
\begin{split}
\mathcal{P}_{21,1}(x')&=\frac{1}{3(\varepsilon+|x'|^2)^3}\Big(\frac{2\lambda+3\mu}{\lambda+2\mu}(\varepsilon+|x'|^2-4x_\alpha^2)+(\varepsilon+|x'|^2-4x_\beta^2) \Big),\\
 \mathcal{Q}_{21,1}(x')&=-\frac{4(\lambda+\mu)}{3(\lambda+2\mu)}\frac{x_1x_2}{(\varepsilon+|x'|^2)^3},\end{split}
\end{align*} $\mathcal{P}_{l1,i}(x')$, $\mathcal{Q}_{l1,i}(x')$, and  $\mathcal{P}_{l2,i}(x')$ are defined in \eqref{3ddtgs1}, \eqref{3ddtgs2} and \eqref{3ddtgs3}, respectively. 

Denote ${\bf v}_\alpha^l(x):={\bf v}_{\alpha}^{l1}(x)+{\bf v}_{\alpha}^{l2}(x)$, $\alpha=1,2$. Then we have \begin{equation*}
\nabla^{m}{\bf u}_{1\alpha}(x)=\nabla^{m}\sum_{l=1}^{m} {\bf v}_{\alpha}^{l}(x)+O(1), \quad x\in\Omega_{R}.
\end{equation*}

For ${\bf u}_{13}$, we define, for $l\ge 2$, $0\le i\le l-2$,
\begin{align}\label{3ddtgs4}
	\mathcal{P}_{l1,i+1}(x')=&\,\frac{(\varepsilon+|x'|^2)^2}{4}\mathcal{P}_{l1,i}(x')\nonumber\\
	&-\frac{1}{a_{l,i}(a_{l,i}-1)(\lambda+2\mu)}\Big((\lambda+\mu)(a_{l,i}-1)\big(\partial_{x_1}\tilde{\mathcal{P}}_{(l-1)2,i+1}(x')\nonumber\\
	&\quad\quad\quad\quad\quad\quad\quad\quad\quad\quad+\partial_{x_2}\tilde{\mathcal{Q}}_{(l-1)2,i+1}(x')\big)+\mu\Delta_{x'}\tilde{\mathcal{P}}_{(l-1)1,i+1}(x')\Big),
\end{align} 
where $a_{l,i}=2(l-i)-1$, while for $0\le i\le l-1$, $\mathcal{P}_{l2,i+1}(x')$ and $\mathcal{Q}_{l2,i+1}(x')$ satisfy
\begin{align}\label{3ddtgs5}
	&\mathcal{P}_{l2,i+1}(x')=\frac{(\varepsilon+|x'|^2)^2}{4}\mathcal{P}_{l2,i}(x')-\frac{1}{a_{l,i}(a_{l,i}+1)\mu}\Big((\lambda+\mu)a_{l,i}\partial_{x_1}\tilde{\mathcal{P}}_{l1,i+1}(x')+(\lambda+2\mu)\nonumber\\
	&\quad\quad\quad\quad~\cdot\partial_{x_1x_1}\tilde{\mathcal{P}}_{(l-1)2,i+1}(x')+ (\lambda+\mu) \partial_{x_1x_2} \tilde{\mathcal{Q}}_{(l-1)2,i+1}(x')+\mu \partial_{x_2x_2}\tilde{\mathcal{P}}_{(l-1)2,i+1}(x')\Big),
\end{align}
\begin{align}\label{3ddtgs6111}
	&\mathcal{Q}_{l2,i+1}(x')=\frac{(\varepsilon+|x'|^2)^2}{4}\mathcal{Q}_{l2,i}(x')-\frac{1}{a_{l,i}(a_{l,i}+1)\mu}\Big((\lambda+\mu)a_{l,i}\partial_{x_2}\tilde{\mathcal{P}}_{l1,i+1}(x')+(\lambda+2\mu)\nonumber\\
	&\quad\quad\quad\quad\quad~\cdot\partial_{x_2x_2}\tilde{\mathcal{Q}}_{(l-1)2,i+1}(x')+(\lambda+\mu) \partial_{x_1x_2} \tilde{\mathcal{Q}}_{(l-1)2,i+1}(x')+\mu \partial_{x_1x_1}\tilde{\mathcal{Q}}_{l1,i+1}(x')\Big).
\end{align}
Here we use the convention that $\mathcal{P}_{l1,i}(x')=0$ if $i\notin\{1,2,\dots,l-1\}$ and $\mathcal{P}_{l2,i}(x')=\mathcal{Q}_{l2,i}(x')=0$ if $i\notin\{1,2,\dots,l\}$.
In $\Omega_{2R}$, let
	\begin{align*}
		{\bf v}_{3}^{11}(x)=\Big(\frac{x_3}{\varepsilon+|x'|^2}+\frac{1}{2}\Big) \boldsymbol{\psi}_{3}, \quad {\bf v}_{3}^{12}(x) =\Big(\mathcal{P}_{12,1}(x'){\boldsymbol\psi}_{1}+\mathcal{Q}_{12,1}(x'){\boldsymbol\psi}_{2}\Big)\Big(x_3^2-\frac{(\varepsilon+|x'|^2)^2}{4}\Big),
	\end{align*}
	where $\mathcal{P}_{12,1}(x')=\frac{\lambda+\mu}{\mu}\frac{x_1}{(\varepsilon+|x'|^2)^2},~\mathcal{Q}_{12,1}= \frac{\lambda+\mu}{\mu}\frac{x_2}{(\varepsilon+|x'|^2)^2}$, and for $l\ge 2$,
	\begin{align*}
\begin{split}
		{\bf v}_{3}^{l1}(x) &=\sum_{i=1}^{l-1}\mathcal{P}_{l1,i}(x')x_3^{2l-1-2i}\Big(x_3^2-\frac{(\varepsilon+|x'|^2)^2}{4}\Big) {\boldsymbol\psi}_{3},\\
		{\bf v}_{3}^{l2}(x)&=\sum_{i=1}^{l}\Big(\mathcal{P}_{l2,i}(x'){\boldsymbol\psi}_{1}+\mathcal{Q}_{l2,i}(x'){\boldsymbol\psi}_{2}\Big)x_3^{2l-2i}\Big(x_3^2-\frac{(\varepsilon+|x'|^2)^2}{4}\Big),
\end{split}
\end{align*}
	where $\mathcal{P}_{21,1}(x')= -\frac{2\lambda}{3}\frac{\varepsilon+|x'|^2-2|x'|^2}{(\varepsilon+|x'|^2)^3}$, $\mathcal{P}_{l1,i}(x')$, $\mathcal{P}_{l2,i}(x')$, and $\mathcal{Q}_{l2,i}(x')$ are defined by \eqref{3ddtgs4}, \eqref{3ddtgs5}, and \eqref{3ddtgs6111}, respectively. Denote ${\bf v}_3^l(x):={\bf v}_{3}^{l1}(x)+{\bf v}_{3}^{l2}(x)$. Then
	\begin{equation*}
\nabla^{m}{\bf u}_{13}(x)=\nabla^{m}\sum_{l=1}^{m}{\bf v}_{3}^{l}(x)+O(1), \quad x\in\Omega_{R}.
	\end{equation*} 

We will use the following asymptotics results from \cite[Proposition 3.7]{LX23}.
\begin{lemma}\label{Sexsjjzk3d}
     Let $C_i^\alpha$ be defined in \eqref{decom_u3d}. The following assertions hold:
	\begin{align*}
		C_{1}^{\alpha}-C_{2}^{\alpha}&=\frac{1}{\pi\mu}\cdot\frac{b_{1\alpha}^{*}[\boldsymbol{\varphi}]}{|\log\varepsilon|}\Big(1+O(\frac{1}{|\log\varepsilon|})\Big),~ \alpha=1,2,~\\
		C_{1}^{3}-C_{2}^{3}&=\frac{1}{\pi(\lambda+2\mu)}\cdot\frac{b_{13}^{*}[\boldsymbol{\varphi}]}{|\log\varepsilon|}\Big(1+O(\frac{1}{|\log\varepsilon|})\Big).
	\end{align*}
\end{lemma}

By using Lemma \ref{Sexsjjzk3d}, we derive the asymptotic expansion for the high-order derivatives as follows. 
\begin{theorem}\label{mainth23d}
	Let $D_1,D_2\subset D$, $\boldsymbol{\varphi}$ be defined as above and let ${\bf u}\in 
	H^{1}(D;\mathbb R^{3})\cap C^{m}(\bar{\Omega};\mathbb R^{3})$ be a solution 
	to \eqref{maineqn} in dimension three. Then for $m\ge 1$ and 
	sufficiently small $0<\varepsilon<1/2$, for $x\in\Omega_{R}$ we have
	\begin{align*}
		\nabla^{m}{\bf u}(x)&=\frac{1}{\pi|\log\varepsilon|}\bigg(\sum_{\alpha=1}^{2}\frac{b_{1\alpha}^{*}[\boldsymbol{\varphi}]}{\mu}\Big(\sum_{l=1}^{m}\nabla^{m} {\bf v}_{\alpha}^{l}(x)\Big)+\frac{b_{13}^{*}[\boldsymbol{\varphi}]}{\lambda+2\mu}\Big(\sum_{l=1}^{m}\nabla^{m}{\bf v}_{3}^{l}(x)\Big)\bigg)\\ &\quad\cdot\big(1+O(|\log\varepsilon|^{-1})\big)+O(1),
	\end{align*}
	where ${\bf v}_{\alpha}^{l}$, $\alpha=1,2,3$, are defined as above. They satisfy
	$$\frac{1}{C(\varepsilon+|x'|^2)^{(m+1)/2}}\leq|\nabla^{m} \sum_{l=1}^{m}{\bf v}_{\alpha}^{l}(x)|\leq\frac{C}{(\varepsilon+|x'|^2)^{(m+1)/2}},\quad m\ge 1.
	$$
\end{theorem}

Similar to Corollary \ref{cor1}, we have
\begin{corollary}
	Under the assumptions of Theorem \ref{mainth23d}, let ${\bf u}\in H^{1}(D;\mathbb R^{3})\cap C^{m}(\bar{\Omega};\mathbb R^{3})$ be a solution to \eqref{maineqn}. For sufficiently small $0<\varepsilon<1/2$, for $x\in\Omega_{R}$ we have
	
	(i) for the gradient matrix and $i=1,2$, 
	\begin{align*}
		\partial_{x_3}{\bf u}^{(i)}(x)&=\frac{b_{1i}^{*}[\boldsymbol{\varphi}]}{\pi\mu}\cdot\frac{1}{|\log\varepsilon|(\varepsilon+|x'|^2)}\Big(1+O(\frac{1}{|\log\varepsilon|})\Big)+O(1),\\
		\partial_{x_3}{\bf u}^{(3)}(x)&=\frac{b_{13}^{*}[\boldsymbol{\varphi}]}{\pi(\lambda+2\mu)}\cdot\frac{1}{|\log\varepsilon|(\varepsilon+|x'|^2)}\Big(1+O(\frac{1}{|\log\varepsilon|})\Big)+O(1),
	\end{align*}
	while the other terms are bounded by a constant $C$;
 
	(ii) for the second-order derivative matrix and $i=1,2$, 
	\begin{align*}
		\partial_{x_ix_3}{\bf u}^{(i)}&=-\frac{2b_{1i}^{*}[\boldsymbol{\varphi}]}{\pi\mu}\cdot\frac{x_i}{|\log\varepsilon|(\varepsilon+|x'|^2)^2}\Big(1+O(\frac{1}{|\log\varepsilon|})\Big)+\frac{O(1)}{|\log\varepsilon|(\varepsilon+|x'|^2)}+O(1),\\
		\partial_{x_ix_3}{\bf u}^{(3)}&=-\frac{2b_{13}^{*}[\boldsymbol{\varphi}]}{\pi(\lambda+2\mu)}\cdot\frac{x_i}{|\log\varepsilon|(\varepsilon+|x'|^2)^2}\Big(1+O(\frac{1}{|\log\varepsilon|})\Big)+\frac{O(1)}{|\log\varepsilon|(\varepsilon+|x'|^2)}+O(1),\\
		\partial_{x_3x_3}{\bf u}^{(i)}&=\frac{2(\lambda+\mu)}{\pi\mu(\lambda+2\mu)}\cdot\frac{b_{13}^{*}[\boldsymbol{\varphi}]x_i}{|\log\varepsilon|(\varepsilon+|x'|^2)^2}\Big(1+O(\frac{1}{|\log\varepsilon|})\Big)+\frac{O(1)}{|\log\varepsilon|(\varepsilon+|x'|^2)}+O(1),\\
		\partial_{x_3x_3}{\bf u}^{(3)}&=\frac{2(\lambda+\mu)}{\pi\mu(\lambda+2\mu)}\cdot\frac{b_{11}^{*}[\boldsymbol{\varphi}]x_1+b_{12}^{*}[\boldsymbol{\varphi}]x_2}{|\log\varepsilon|(\varepsilon+|x'|^2)^2}\Big(1+O(\frac{1}{|\log\varepsilon|})\Big)+\frac{O(1)}{|\log\varepsilon|(\varepsilon+|x'|^2)}+O(1),
	\end{align*}
while the other terms are bounded by $\frac{C}{|\log\varepsilon|(\varepsilon+|x'|^2)}+C.$
\end{corollary}

\subsection{Lower bounds for the higher derivatives}
To demonstrate the optimality of the upper bounds obtained in Theorems \ref{mainth} and \ref{mainth3d}, we present a lower bound result at the end of this section.

\begin{theorem}\label{thm-5.12}
	Let $D_1,D_2\subset D$ be defined as above and let ${\bf u}\in 
H^{1}(D;\mathbb R^{d})\cap C^{m}(\bar{\Omega};\mathbb R^{d})$ be a solution 
to \eqref{maineqn}. If there exists a $\boldsymbol{\varphi}$ such that $\boldsymbol{\varphi}(-x)=-\boldsymbol{\varphi}(x)$ and $b_{11}^{*}[\boldsymbol{\varphi}]\neq 0$. Then for $m\ge 1$ and 
sufficiently small $0<\varepsilon<1/2$, there exists a small constant $r>0$ which may depend on $m$, such that
\begin{align*}
|\partial_{x_1}^{m-1}\partial_{x_{d}}{\bf u}^{(1)}(r\sqrt{\varepsilon},0)|\ge C|b_{11}^{*}[\boldsymbol{\varphi}]|\varepsilon^{-m/2}\quad \text{for}~ d=2,
\end{align*}
and 
\begin{align*}
	|\partial_{x_1}^{m-1}\partial_{x_{d}}{\bf u}^{(1)}(r\sqrt{\varepsilon},0,\dots,0 )|\ge C|b_{11}^{*}[\boldsymbol{\varphi}]||\log\varepsilon|^{-1}\varepsilon^{-(m+1)/2}\quad \text{for}~ d=3.
\end{align*}
\end{theorem}

\begin{proof}[Proof of Theorem \ref{thm-5.12}]
	We only prove the case when $d=2$ for instance, since the case when $d=3$ is similar. By Theorem \ref{mainth-5.3}, we have
	\begin{align*}
			&\partial_{x_1}^{m-1}\partial_{x_{2}}{\bf u}^{(1)}\\
   =&\frac{\sqrt{\varepsilon}}{\pi}\Big(\frac{b_{11}^{*}[\boldsymbol{\varphi}]}{\mu}	\partial_{x_1}^{m-1}\partial_{x_{2}} \sum_{l=1}^{m}({\bf v}_{1}^{l})^{(1)}+\frac{b_{12}^{*}[\boldsymbol{\varphi}]}{\lambda+2\mu}	\partial_{x_1}^{m-1}\partial_{x_{2}} \sum_{l=1}^{m}({\bf v}_{2}^{l})^{(1)}\Big)(1+O(\varepsilon^{1/4}))+O(1).
	\end{align*}
	For small $\varepsilon>0$, $\frac12\leq 1+O(\varepsilon^{1/4})\le \frac32$, thus, 
	\begin{align}\label{gjzz1}
		|\partial_{x_1}^{m-1}\partial_{x_{2}}{\bf u}^{(1)}|
		\ge&\, C\sqrt{\varepsilon}|b_{11}^{*}[\boldsymbol{\varphi}]|\partial_{x_1}^{m-1}\partial_{x_{2}}({\bf v}_{1}^{1})^{(1)}\nonumber\\
		&-C\sqrt{\varepsilon}|\partial_{x_1}^{m-1}\partial_{x_{2}}({\bf v}_{2}^{1})^{(1)}|-C\sqrt{\varepsilon} \Big|\sum_{l=2}^m \sum_{i=1}^{2}\partial_{x_1}^{m-1}\partial_{x_{2}}({\bf v}_{i}^{l})^{(1)}\Big|-C.
	\end{align}
	By \eqref{def2dv11i} and \eqref{def2dvi1i}, it is easy to verify that, for $l\ge 1$, $i,j=1,2$, $i\neq j$,
	\begin{equation*}
			|({\bf v}_{i}^{l})^{(i)}(x) | \le C\delta(x_1)^{l-1},\quad\quad |({\bf v}_{i}^{l})^{(j)}(x) | \le C\delta(x_1)^{l-1/2},
	\end{equation*}
where $\delta(x_1)=\varepsilon+x_1^2$. In addition, we have
\begin{equation}\label{lowergj}
|\partial_{x_1}^{m-1}\partial_{x_{2}}({\bf v}_{i}^{l})^{(i)}|\leq C\delta(x_1)^{(2l-m-3)/2},\quad\quad |\partial_{x_1}^{m-1}\partial_{x_{2}}({\bf v}_{i}^{l})^{(j)} | \leq C\delta(x_1)^{(2l-m-2)/2}.
\end{equation}
It follows from \eqref{def2dv11i} that for some small $r>0$, which may depend on $m$,
\begin{equation}\label{lower1}
|\partial_{x_1}^{m-1}\partial_{x_{2}}({\bf v}_{1}^{1})^{(1)}(r\sqrt{\varepsilon},0 )|=|\partial_{x_1}^{m-1}(\delta(x_1)^{-1})(r\sqrt{\varepsilon},0 )|\ge C\varepsilon^{-(m+1)/2},
\end{equation}
where $C>0$. By \eqref{lowergj}, we obtain
\begin{equation}\label{lower2}
|\partial_{x_1}^{m-1}\partial_{x_{2}}({\bf v}_{2}^{1})^{(1)}|\leq C\varepsilon^{-m/2},\quad\quad\Big|\sum_{l=2}^m \sum_{i=1}^{2}\partial_{x_1}^{m-1}\partial_{x_{2}}({\bf v}_{i}^{l})^{(1)}\Big|\leq C\varepsilon^{(1-m)/2}.
\end{equation}
Combining \eqref{lower1} and \eqref{lower2} with \eqref{gjzz1} yields, for small $\varepsilon>0$,
\begin{align*}
|\partial_{x_1}^{m-1}\partial_{x_{d}}{\bf u}^{(1)}(r\sqrt{\varepsilon},0)|&\ge C|b_{11}^{*}[\boldsymbol{\varphi}]|\varepsilon^{-m/2}-C\varepsilon^{(2-m)/2}-C\varepsilon^{(1-m)/2}-C\\
&\ge C|b_{11}^{*}[\boldsymbol{\varphi}]|\varepsilon^{-m/2}.
\end{align*}
We thus complete the proof.
\end{proof}


\section{Lam\'e Systems with Two Close Holes}\label{sec6}

In this section, we prove Theorem \ref{mainthinsl} and establish an upper bound estimate for the higher-order derivatives of the solution to the Lam\'e system \eqref{maineqn222} with two closely spaced holes. We follow the method used in \cite{LY} , where Li and Yang  obtained the following gradient estimate in all dimensions,
\begin{align}\label{insludsgj}
|\nabla {\bf u}(x)|\leq C\|{\bf u}\|_{L^\infty (\Omega_{R})}(\varepsilon+|x'|^2)^{-1/2}, \quad x\in \Omega_{R}.
\end{align}

\begin{proof}[Proof of Theorem \ref{mainthinsl}]
For a fixed point $x_0\neq 0\in \Omega_{R}$, by means of a change of variables
\begin{align}\label{zbbh01}
	\begin{cases}
		y'=\frac{1}{\sqrt{\delta(x_0')}}(x'-x_0'),\\
		y_d=\frac{1}{\sqrt{\delta(x_0')}}x_d,
	\end{cases}
\end{align}
where $\delta(x_0')=\varepsilon+|x_0'|^2$, we transform the narrow region $\Omega_{\sqrt{\delta(x_0')}/4}(x_0')$ to
\begin{align}\label{domain1}
	\Big\{y\in\mathbb{R}^d: |y'| \le \frac{1}{4}, \frac{-1}{\sqrt{\delta}}\Big(\frac{\varepsilon}{2}+h_2\big(x_0'+\sqrt{\delta}y'\big)\Big)< y_d <\frac{1}{\sqrt{\delta}}\Big(\frac{\varepsilon}{2}+h_1\big(x_0'+\sqrt{\delta}y'\big)\Big) \Big\},
\end{align}
where $\delta:=\delta(x_0')$ and
\begin{equation*}
	\Omega_{r}(x_0'):=\left\{(x',x_{d})\in \mathbb{R}^{d}:~-\frac{\varepsilon}{2}-h_{2}(x')<x_{d}<\frac{\varepsilon}{2}+h_{1}(x'),~|x'-x_0'|<r\right\},~0\le r\le R.
\end{equation*}
We use another change of variables by
\begin{align}\label{zbbh11}
\begin{cases}
	z'=4y',\\
	z_d=2\sqrt{\delta(x_0')}\Big(\frac{\sqrt{\delta(x_0')}y_d+h_2\big(x_0'+\sqrt{\delta(x_0')}y'\big)+{\varepsilon}/{2}}{\delta(x_0'+\sqrt{\delta(x_0')}y')}-\frac12\Big),
\end{cases}
\end{align}
which transforms the region defined by \eqref{domain1} to a cylinder $Q_{1,\sqrt{\delta(x_0')}}$, where 
$$Q_{s,t}:=\{z=(z',z_d):~|z'|\le s,~|z_d|\le t\},$$
for $s,t>0$. Let ${\bf v}(z)={\bf u}(x)$. Then ${\bf v}(z)$ satisfies
\begin{align}\label{changevariable}
	\begin{cases}
		\partial_{\theta}\left(B_{ij}^{\theta\gamma}(z)\partial_{\gamma}{\bf v}^j(z)\right)=0&\mbox{in}~Q_{1,\sqrt{\delta(x_0')}},\\
	    B_{ij}^{d\gamma}(z)\partial_{\gamma}{\bf v}^j(z)=0&\mbox{on}~\{y_d=\pm\sqrt{\delta(x_0')}\}.
	\end{cases}
\end{align}
The coefficient matrix $( B_{ij}^{\theta\gamma}(z))$ is given by
\begin{align}\label{insufz1}
B_{ij}^{\theta\gamma}=A_{ij}^{\alpha\beta}\frac{\partial z^{\theta}}{\partial y^{\alpha}}\frac{\partial z^{\gamma}}{\partial y^{\beta}}\det\partial_zy,
\end{align}
where 
$A^{\alpha\beta}_{ij}=\lambda\delta_{\alpha i}\delta_{\beta j}+\mu(\delta_{\alpha\beta}\delta_{ij}+\delta_{\alpha j}\delta_{i\beta}),$ the Jacobian matrix of the change of variables \eqref{zbbh11} is denoted by $\partial_y z$, and its inverse matrix  is denoted by  $\partial_z y$.

By using \eqref{zbbh11}, we derive, for $1\le i\le d-1,$ and $j\neq i$,
\begin{align*}
(\partial_y z)^{ii}&=4,\quad\quad\quad (\partial_y z)^{dd}=2\delta(x_0')\delta\big(x_0'+\sqrt{\delta(x_0')}z'/4\big)^{-1},\quad\quad(\partial_x y)^{ij}=0,\\
(\partial_y z)^{di}&=\frac{2\delta(x_0')\partial_{x_{i}}h_2\big(x_0'+\sqrt{\delta(x_0')}z'/4\big)-\sqrt{\delta(x_0')}\big(z_d+\sqrt{\delta(x_0')}\big)\partial_{x_{i}}\delta\big(x_0'+\sqrt{\delta(x_0')}z'/4\big)}{\delta\big(x_0'+\sqrt{\delta(x_0')}z'/4\big)}.
\end{align*}
Moreover, since $|z_d|\leq \sqrt{\delta(x_0')}$, $z'\le 1$, and $|x_0'|\le \sqrt{\delta(x_0')}$, by \eqref{h1h1} and \eqref{h1h14}, we have
\begin{align*}
	|(\partial_y z)^{di}|\leq \frac{C\delta(x_0')\big|x_0'+\sqrt{\delta(x_0')}z'/4\big|}{\delta(x_0')(1+(1/4)^2|z'|^2-(1/2)|z'|)}\leq C\sqrt{\delta(x_0')}.
\end{align*}
We further require $R$ to be sufficiently small so that $(\partial_y z)^{di}$ are small  enough. Since $|z'|<1$ and $|x_0'|<\sqrt{\delta(x_0')}$, we have
\begin{align*}
(\partial_y z)^{dd}\ge\frac{2\delta(x_0')}{\varepsilon+|x_0'+\sqrt{\delta(x_0')}z'/4|^2} \ge \frac{2\delta(x_0')}{\varepsilon+C\delta(x_0')}\ge \frac{1}{C},\quad z\in Q_{1,\sqrt{\delta(x_0')}}.
\end{align*}
In addition, we have
\begin{align*}
(\partial_y z)^{dd}&\le \frac{C\delta(x_0')}{\varepsilon+|x_0'|^2+(1/4)^2\delta(x_0')|z'|^2+(1/2)|x_0'\cdot z'|\sqrt{\delta(x_0')}}\\& \leq \frac{C\delta(x_0')}{\delta(x_0')+(1/4)^2\delta(x_0')|z'|^2-(1/2)|x_0'||z'|\sqrt{\delta(x_0')}}\\
&\le\frac{C\delta(x_0')}{\delta(x_0')(1+(1/4)^2|z'|^2-(1/2)|z'|)}\le C,\quad z\in Q_{1,\sqrt{\delta(x_0')}}.
\end{align*}
This leads to 
\begin{align*}
\frac{1}{C}\le\det\partial_y z,~\det\partial_z y\le C.
\end{align*}

By using \eqref{insufz1}, we obtain
\begin{align*}
	\frac 1 C\le|B_{ij}^{\theta\gamma} (z)|\leq C,\quad z\in Q_{1,\sqrt{\delta(x_0')}}.
\end{align*}
For $m\ge 0$, we have
\begin{align*}
|\partial_{z'}^m(\partial_y z)^{dd}|\le C\Big|\frac{\delta(x_0')^{(m+2)/2}|x_0'+\sqrt{\delta(x_0')}z'|^m)}{\delta(x_0'+\sqrt{\delta(x_0')}z'/4)^{m+1}}\Big|\leq C, \quad z\in Q_{1,\sqrt{\delta(x_0')}}.
\end{align*}
Similar computations show that
\begin{align*}
|\partial_{z'}^m(\partial_y z)^{dd}|,~|\partial_{z}^m(\partial_y z)^{di}|\leq C.
\end{align*}
Thus, for $0<\mu<1$,
\begin{align*}
\|\partial_{z}^mB_{ij}^{\theta\gamma}\|_{C^\mu(\bar{Q}_{1,\sqrt{\delta(x_0')}})}\leq C.
\end{align*}

Now, for integer $l$, we define
$$S_l:= \left\{z=(z',z_d) :~   |z'| < 1,~ (2l-1) \sqrt{\delta(x_0')} < z_d < (2l+1) \sqrt{\delta(x_0')} \right\}$$
and 
$$S:= \left\{z=(z',z_d) :~   |z'| < 1,~|z_d| < 1 \right\}.$$
Obviously, $Q_{1,\sqrt{\delta(x_0')}}=S_0$ and $Q_{1,1}=S$. We will extend the definitions of ${\bf v}(z)$ and $B_{ij}^{\theta\gamma}(z)$ from $S_0$ to the entire $S$, and still denote them as ${\bf v}(z)$ and $B_{ij}^{\theta\gamma}(z)$ if it does not cause confusion. For $l\neq 0$, we denote
$${\bf v}(z',z_d) := {\bf v}\Big(z', (-1)^l\Big(z_d- 2l \sqrt{\delta(x_0')}\Big)\Big), \quad \forall~ z \in S_l,$$
and define the corresponding coefficients, for $k = 1,2, \dots, d-1$,
$$ B_{ij}^{dk}(z)= B_{ij}^{kd}(z) := (-1)^l B_{ij}^{dk}\Big(z', (-1)^l\Big(z_d - 2l \sqrt{\delta(x_0')}\Big)\Big),  \quad \forall ~z \in S_l,$$
and for other indices,
$$ B_{ij}^{\theta\gamma}(z) := B_{ij}^{\theta\gamma}\Big(z', (-1)^l\Big(z_d - 2l \sqrt{\delta(x_0')}\Big)\Big), \quad \forall~ z \in S_l.$$
By \eqref{changevariable}, ${\bf v}(z)$ satisfies
\begin{align}\label{changevariable1}
		\partial_{\theta}\left(B_{ij}^{\theta\gamma}(z)\partial_{\gamma}{\bf v}^j(z)\right)=0\quad\mbox{in}~S.
\end{align}
By using \cite[Theorem 1]{D2012}, we have
\begin{align}\label{dyb1}
	\|\nabla{\bf v}\|_{C^\mu(Q_{1/2,\sqrt{\delta(x_0)}})}\leq C\|{\bf v}\|_{L^2(S)}\leq  C\|{\bf v}(z)\|_{L^\infty(Q_{1, \sqrt{\delta(x_0')}})}.
\end{align}

For $k=1,2,\dots, d-1$, differentiating both sides of equation \eqref{changevariable1} with respect to $z_k$, we have
\begin{align}\label{changevariable2}
		\partial_{\theta}\left(B_{ij}^{\theta\gamma}(z)\partial_{\gamma}(\partial_k{\bf v}^j(z))\right)=	-\partial_{\theta}\left(\partial_kB_{ij}^{\theta\gamma}(z)\partial_{\gamma}{\bf v}^j(z)\right)\quad\mbox{in}~S.
\end{align}
Thus, by using \cite[Theorem 1]{D2012} and \eqref{dyb1} (see also \cite[Lemma 2.1]{LY}), we have
\begin{align*}
\|\nabla (\partial_k{\bf v})\|_{C^\mu(Q_{1/4, \sqrt{\delta(x_0')}})}\leq&\, C\|\partial_k {\bf v}\|_{L^2(Q_{S})}+C\|\partial_kB_{ij}^{\theta\gamma}(z)\partial_{\gamma}{\bf v}^j\|_{C^\mu(Q_{1/2, \sqrt{\delta(z')}})}\\
\leq&\, C\|\partial_k {\bf v}\|_{L^\infty(Q_{1, \sqrt{\delta(x_0')}})}+C\|{\bf v}\|_{L^\infty(Q_{1, \sqrt{\delta(x_0')}})}.
\end{align*}

Reversing the change of variables \eqref{zbbh11} and \eqref{zbbh01}, we deduce
\begin{align*}
\delta(x_0')\|\nabla (\partial_k {\bf u})\|_{L^\infty(\Omega_{\sqrt{\delta(x_0')}/4}(x_0'))}\leq C\sqrt{\delta(x_0')}\|\nabla {\bf u}\|_{L^\infty(\Omega_{\sqrt{\delta(x_0')}/4}(x_0'))}+C\|{\bf u}\|_{L^\infty(\Omega_{\sqrt{\delta(x_0')}/4}(x_0'))}.
\end{align*}
Together with \eqref{insludsgj}, we obtain
\begin{align*}
\|\nabla (\partial_k {\bf u})\|_{L^\infty(\Omega_{\sqrt{\delta(x_0')}/4}(x_0'))}\leq C\|{\bf u}\|_{L^\infty (\Omega_{R})}\delta(x_0')^{-1}.
\end{align*}
By using the first line of \eqref{maineqn222}, since $\lambda,\mu> 0$, we have, for $j\neq d$,
\begin{align*}
&\|\partial_{x_dx_d}{\bf u}^j\|_{L^\infty(\Omega_{\sqrt{\delta(x_0')}/4}(x_0'))}\\
\leq&\,\|\Delta_{x'}{
\bf u}^j\|_{L^\infty(\Omega_{\sqrt{\delta(x_0')}/4}(x_0'))} +\frac{\lambda+\mu}{\mu}\|\partial_{x_j}(\nabla\cdot {\bf u})\|_{L^\infty(\Omega_{\sqrt{\delta(x_0')}/4}(x_0'))}\\
\leq&\, C\|{\bf u}\|_{L^\infty (\Omega_{R})}\delta(x_0')^{-1}
\end{align*}
and 
\begin{align*}
	&\|\partial_{x_dx_d}{\bf u}^d(x)\|_{L^\infty(\Omega_{\sqrt{\delta(x_0')}/4}(x_0'))}\\
	\leq&\, C\|\Delta_{x'}{
		\bf u}^d\|_{L^\infty(\Omega_{\sqrt{\delta(x_0')}/4}(x_0'))} +C\Big\|\sum_{i=1}^{d-1}\partial_{x_ix_d}{\bf u}^i\Big\|_{L^\infty(\Omega_{\sqrt{\delta(x_0')}/4}(x_0'))}\\
	\leq&\, C\|{\bf u}\|_{L^\infty (\Omega_{R})}\delta(x_0')^{-1}.
\end{align*}
Thus, we have the following second-order derivative estimate
$$\|\nabla^2{\bf u}(x)\|_{L^\infty(\Omega_{\sqrt{\delta(x_0')}}(x_0'))}\leq C\delta(x_0')^{-1}.$$
Continuing to differentiate both sides of equation \eqref{changevariable2} with respect to $z_k$ for $k=1,2,\dots, d-1$ and repeating the previous process, we obtain \eqref{inslut1}. We thus complete the proof.
\end{proof}


\section{Appendix: A Framework for the Energy Method}

In this section, we present some framework results for the energy method and demonstrate that the auxiliary functions constructed in Sections \ref{sec2} and \ref{sec4} can capture all singular terms in $\nabla^{m}{\bf u}_{1\alpha}$ in $\Omega_{R}$.  Let 
$${\bf w}_{1\alpha}^{m}:={\bf u}_{1\alpha}-{\bf v}_{\alpha}^{m}, \quad\alpha=1,\cdots,d(d+1)/2,\quad\mathrm{in}~\Omega_{2R},$$ 
where ${\bf u}_{1\alpha}={\bf v}_{\alpha}^{m}$ on $\Gamma^+_{2R}\cup\Gamma^-_{2R}$. Then ${\bf w}_{1\alpha}^{m}$ verifies
\begin{equation}\label{wia}
\begin{cases}
\mathcal{L}_{\lambda,\mu}{\bf w}_{1\alpha}^{m}={\bf f}_{\alpha}^{m}:=-\mathcal{L}_{\lambda,\mu}{\bf v}_{\alpha}^{m}&\mathrm{in}~\Omega_{2R},\\
{\bf w}_{1\alpha}^{m}=0&\mathrm{on}~\Gamma^+_{2R}\cup\Gamma^-_{2R}.
\end{cases}
\end{equation}

We have 
\begin{prop}\label{wlinfty}
Let ${\bf w}_{1\alpha}^{m}$ be the solution to \eqref{wia}. Then for $z=(z',z_d)\in\Omega_R$ and $m\geq 1$, the following estimate holds
\begin{align}\label{keyest}
&\|\nabla^{m}{\bf w}_{1\alpha}^{m}\|_{L^\infty(\Omega_{\delta(z')/2}(z'))}\nonumber\\
\leq&\, \frac{1}{\delta(x')^m}\Big(\frac{1}{\delta(z')^{(d-2)/2}}\|\nabla{\bf w}_{1\alpha}^{m}\|_{L^2(\Omega_{\delta(z')/2}(z'))}+\sum_{j=0}^{m-1}\delta(z')^{2+j}\|\nabla^j{\bf f}_{\alpha}^{m}\|_{L^\infty(\Omega_{\delta(z')}(z'))}\Big).
\end{align}
\end{prop}

For the sake of simplicity of notation, in the following we will omit superscripts and subscripts when they do not cause confusion. 
\begin{proof}
We take a change of variables, as in \cite{bll1,bll2},
\begin{equation*}
\left\{
\begin{aligned}
&x'-z'=\delta(z') y',\\
&x_d=\delta(z') y_{d},
\end{aligned}
\right.
\end{equation*}
to transform $\Omega_{\delta(z')}(z')$ into a nearly unit size domain $Q_{1}$, where 
\begin{align*}
Q_{r}=\left\{y\in\mathbb{R}^{d}:~-\frac{\varepsilon}{2\delta}-\frac{1}{\delta}h_{2}(\delta\,y'+z')<y_{d}
<\frac{\varepsilon}{2\delta}+\frac{1}{\delta}h_{1}(\delta\,y'+z'),~|y'|<r\right\},~\delta=\delta(z')
\end{align*}
for $0<r<1$. Denote its upper and lower boundaries by $\Gamma_{1}^{+}$ and $\Gamma_{1}^{-}$, respectively. For $y\in Q_1$,let us define
\begin{equation*}
\mathcal{W}(y', y_{d}):={\bf w}(\delta\,y'+z',\delta\,y_{d}),~ \mathcal{V}(y', y_{d}):={\bf v}(\delta\,y'+z',\delta\,y_{d}),~\text{and}~\mathcal{F}(y', y_{d}):={\bf f}(\delta\,y'+z',\delta\,y_{d}).
\end{equation*} 
Then it follows from \eqref{wia} that
\begin{equation*}
\begin{cases}
\mathcal{L}_{\lambda,\mu}\mathcal{W}=\mathcal{F}&\mathrm{in}~Q_1,\\
\mathcal{W}=0&\mathrm{on}~\Gamma_{1}^{+}\cup\Gamma_{1}^{-}.
\end{cases}
\end{equation*}
By using the $W^{k,p}$ estimates for elliptic systems with partially vanishing boundary (see \cite[Theorem 10.6]{adn0}), we derive
\begin{equation*}
\|\nabla^k\mathcal{W}\|_{W^{2,p}(Q_{1/2})}\leq C\Big(\|\nabla\mathcal{W}\|_{L^{2}(Q_1)}+\sum_{j=0}^{k}\|\nabla^j\mathcal{F}\|_{L^p(Q_1)}\Big).
\end{equation*}
Taking $k=m-1$, together with the Sobolev embedding theorem for $p>d$, leads to
\begin{equation*}
\|\nabla^{m}\mathcal{W}\|_{L^{\infty}(Q_{1/2})}\leq C\|\nabla^{m-1}\mathcal{W}\|_{W^{2,p}(Q_{1/2})} \leq C\Big(\|\nabla\mathcal{W}\|_{L^{2}(Q_1)}+\sum_{j=0}^{m-1}\|\nabla^j\mathcal{F}\|_{L^\infty(Q_1)}\Big).
\end{equation*}
Rescaling back to the domain $\Omega_{\delta(z')}(z')$, we have \eqref{keyest}. This finishes the proof.
\end{proof}

Based on Proposition \ref{wlinfty}, it is clear that in order to obtain higher derivative estimates, we need to establish the local energy estimates of ${\bf w}_{1\alpha}^{m}$ and the pointwise upper bounds of ${\bf f}_{\alpha}^m$.  Actually, the local energy estimates of ${\bf w}_{1\alpha}^{m}$ depend on the estimates of ${\bf f}_{\alpha}^m$ in $\Omega_{R}$, as demonstrated in Lemmas \ref{ztnl} and \ref{jbnl} below. We first give the boundedness of the global energy in the domain $\Omega_{R}$, provided that ${\bf f}_{\alpha}^{m}$ satisfies certain assumptions.

\begin{lemma}\label{ztnl}
Let ${\bf w}_{1\alpha}^{m}$ be the solution to \eqref{wia}. If 
\begin{equation}\label{cond_f}
|{\bf f}_{\alpha}^{m}(x)|\leq C\delta(x')^{-3/2}\quad\mbox{in}~ \Omega_{2R},\quad\mbox{and} ~\|{\bf w}_{1\alpha}^m\|_{L^\infty(\Omega_{2R}\setminus\Omega_{R})}\leq C,
\end{equation} then
\begin{align*}
\int_{\Omega_{3R/2}}|\nabla {\bf w}_{1\alpha}^{m}|^2\leq C.
\end{align*}
\end{lemma}

\begin{proof}
Let $\eta$ be a smooth function satisfying $\eta(x')=1$ if $|x'|<3R/2$, $\eta(x')=0$ if $|x'|>2R$, $0\leqslant\eta(x')\leqslant1$ if $3R/2\le |x'|\le 2R$, and $|\nabla_{x'}\eta(x')|\leq C$. Multiplying the equation in \eqref{wia} by ${\bf w}\eta^2$ and integrating by parts, we have
\begin{align}\label{ge1x}
\int_{\Omega_{2R}}(\mathbb{C}^0 e({\bf w}), e({\bf w}\eta^2)) =\int_{\Omega_{2R}}\eta^2{\bf w}\cdot {\bf f}.
\end{align}
By using the first Korn inequality and some standard arguments, we have
\begin{equation}\label{gee2x}
\int_{\Omega_{2R}}(\mathbb{C}^0 e({\bf w}), e({\bf w}\eta^2))\ge \frac{1}{C}\int_{\Omega_{2R}}|\nabla ({\bf w}\eta)|^2-C\int_{\Omega_{2R}}|{\bf w}|^2|\nabla \eta|^2.
\end{equation}
From \eqref{cond_f}, it follows that
\begin{align*}
	\|\delta(x') {\bf f}\|_{L^2(\Omega_{3R/2})}\leq C\Big(\int_{\Omega_{3R/2}}\frac{1}{\delta(x')}\Big)^{1/2}\leq C.
\end{align*}
For the right-hand side of \eqref{ge1x}, by using H{\"o}lder’s inequality and Hardy’s inequality, we derive
\begin{align}\label{ge2x}
	\Big|\int_{\Omega_{2R}}\eta^2{\bf w}\cdot {\bf f}\Big|\leq&\, \Big|\int_{\Omega_{3R/2}}{\bf w}\cdot {\bf f}\Big|+\Big|\int_{\Omega_{2R}\setminus\Omega_{3R/2}}\eta^2{\bf w}\cdot {\bf f}\Big|\nonumber\\
	\leq&\,\|\delta(x') {\bf f}\|_{L^2(\Omega_{3R/2})}\Big\|\frac{\bf w}{\delta(x')}\Big\|_{L^2(\Omega_{3R/2})}+C\nonumber\\
	\leq& C\Big(\int_{\Omega_{3R/2}}|\nabla{\bf w}|^2\Big)^{1/2}+C.
\end{align}
It follows from \eqref{cond_f}--\eqref{ge2x} that
\begin{align*}
\int_{\Omega_{3R/2}}|\nabla {\bf w}|^2\leq&\, C\Big(\int_{\Omega_{3R/2}}|\nabla{\bf w}|^2\Big)^{1/2}+C\int_{\Omega_{2R}\setminus\Omega_{3R/2}}|{\bf w}|^2|\nabla \eta|^2+C\\
\leq&\, C\Big(\int_{\Omega_{3R/2}}|\nabla{\bf w}|^2\Big)^{1/2}+C.
\end{align*}
Hence, $\int_{\Omega_{3R/2}}|\nabla {\bf w}|^2\leq C$. This completes the proof.
\end{proof}

As previously mentioned, if we can construct more auxiliary functions to make $|{\bf f}_{1\alpha}^{m}(x)|$ smaller, then we can get better local energy estimates of ${\bf w}_{1\alpha}^{m}$ by the iteration technique.
\begin{lemma}\label{jbnl}
Let ${\bf w}_{1\alpha}^{m}\in H^{1}(\Omega_{3R/2})$ be the solution to \eqref{wia}. If 
\begin{align}\label{jstj}
|{\bf f}_{1\alpha}^{m}(x)|\leq C\delta(x')^a,  \quad\forall~a\ge -\frac{3}{2},\quad x \in \Omega_{2R},
\end{align}
then, for sufficiently small $0\leq \varepsilon\leq 1/2$, 
\begin{align*}
\int_{\Omega_{\delta(z')}(z')}|\nabla{\bf w}_{1\alpha}^{m}(x)|^2\leq 
C\delta(z')^{2a+d+2},\quad z \in \Omega_{R}.
\end{align*}
\end{lemma}

\begin{proof}
For $0<t<s\leq R/2$, let $\eta$ be a smooth function satisfying $\eta(x')=1$ if $|x'-z'|<t$, $\eta(x')=0$ if $|x'-z'|>s$, $0\leqslant\eta(x')\leqslant1$ if $t\leqslant|x'-z'|<s$,
and $|\nabla_{x'}\eta(x')|\leq \frac{2}{s-t}$. Multiplying the equation in \eqref{wia} by ${\bf w}\eta^2$ and integrating by parts leads to the following Caccioppoli’s type inequality
\begin{align}\label{jbnlfz2}
\int_{\Omega_t(z')}|\nabla {\bf w}|^2\leq \frac{C}{(s-t)^2} \int_{\Omega_s(z')}|{\bf w}|^2+(s-t)^2\int_{\Omega_s(z')}|{\bf f}|^2.
\end{align}
We refer the reader to \cite[p.319. Step 2]{bll1} for more details. Since ${\bf w}=0$ on $\Gamma^{\pm}_{2R}$, employing the Poincar\'e inequality yields (see \cite[(3.36),(3.39)]{bll2}),
\begin{align}\label{estwDwl2}
\int_{\Omega_{s}(z')}|{\bf w}|^2\leq
C\delta^2(z')\int_{\Omega_{s}(z')}|\nabla {\bf w}|^2, ~0< s\le C\sqrt{\delta(z')}.
\end{align}
By using \eqref{jstj}, 
\begin{align}\label{jbnl1}
\int_{\Omega_s(z')}|{\bf f}|^2\leq Cs^{d-1}\delta(z')^{2a+1},~0<s\le C\sqrt{\delta(z')}.
\end{align}
Substituting \eqref{estwDwl2} and \eqref{jbnl1} into \eqref{jbnlfz2} yields
\begin{equation}\label{jbnlfz3}
\int_{\Omega_t(z')}|\nabla {\bf w}|^2\leq \frac{C\delta(z')^2}{(s-t)^2} \int_{\Omega_s(z')}|\nabla{\bf w}|^2+C(s-t)^2s^{d-1}\delta(z')^{2a+1}.
\end{equation}

Set $E(t):=\int_{\Omega_{t}(z_1)}|\nabla {\bf w}|^{2}\mathrm{d}x$. By virtue of  \eqref{jbnlfz3}, we have
\begin{equation}\label{iteration3D}
E(t)\leq \left(\frac{c_{0}\delta(z')}{s-t}\right)^2 E(s)+C(s-t)^2s^{d-1}\delta(z')^{2a+1},
\end{equation}
where $c_{0}$ is a constant and we fix it now. Let $k_{0}=\left[\frac{1}{8c_{0}\sqrt{\delta(z')}}\right]$ and $t_{i}=\delta(z')+2c_{0}i\delta(z'), i=0,1,2,\dots,k_{0}$. Then by \eqref{iteration3D} with $s=t_{i+1}$ and $t=t_{i}$, we have the following iteration formula:
\begin{align*}
E(t_{i})\leq \frac{1}{4}E(t_{i+1})+C(i+1)^{d-1}\delta(z')^{2a+d+2}.
\end{align*}
After $k_{0}$ iterations, and using Lemma \ref{ztnl}, we obtain
\begin{align*}
E(t_0)\leq&\, \left(\frac{1}{4}\right)^{k_0}E(t_{k_0})
+C\delta(z')^{2a+d+2}\sum\limits_{l=0}^{k_0-1}\left(\frac{1}{4}\right)^{j}(j+1)^l\leq C\delta(z')^{2a+d+2},
\end{align*}
for sufficiently small $\varepsilon$. This completes the proof.
\end{proof}

With the help of Proposition \ref{wlinfty} and Lemma \ref{jbnl}, we reduce the higher derivative estimates of ${\bf w}_{1\alpha}^{m}$ to continuously improving the estimates of $|{\bf f}_{\alpha}^{m}(x)|$ and their derivatives.
\begin{prop}\label{propnew}
Let ${\bf w}_{1\alpha}^{m}\in H^{1}(\Omega_{3R/2})\cap L^{\infty}(\Omega_{2R}\setminus\Omega_{R})$ be the solution to \eqref{wia}. If
\begin{equation*}
|{\bf f}_{\alpha}^{m}(x)|\leq C\delta(x')^a,\quad\mbox{and}~|\nabla^k{\bf f}_{\alpha}^{m}(x)|\leq C\delta(x')^{a-k}, \quad\,1\le k\leq l-1,~a\ge -\frac{3}{2}\quad\mbox{for}~x\in\Omega_{2R},
\end{equation*}
then there holds
\begin{align*}
\|\nabla^{l}{\bf w}_{1\alpha}^{m}\|_{L^\infty(\Omega_{\delta(z')/2}(z'))}\leq C\delta(z')^{a+2-l},\quad 1\leq l\leq m,\quad z \in \Omega_{R}.
\end{align*}
\end{prop}
Therefore, in order to show that $\|\nabla^{m}{\bf w}_{1\alpha}^{m}\|_{L^\infty(\Omega_{\delta(z')/2}(z'))}\leq C$, we only need to choose suitable auxiliary functions ${\bf v}_{\alpha}^{m}$ so that the nonhomogeneous terms ${\bf f}_{\alpha}^{m}$ in \eqref{wia} and their derivatives are under control.



\section*{Declarations}

\noindent{\bf Data availability.} Data sharing not applicable to this article as no data sets were generated or analysed during the current study.

\noindent{\bf Financial interests.} The authors have no relevant financial or non-financial interests to disclose.

\noindent{\bf Conflict of interest.} The authors declare that they have no conflict of interest.


\end{document}